\documentclass[11pt]{amsart}

\usepackage[all]{xy}
\usepackage{color, pstricks}
\usepackage{mathrsfs}
\usepackage{calligra}
\usepackage{amsmath,amsthm, amssymb, amsgen,amsxtra,amsfonts, amsbsy} 
\usepackage[all]{xy}
\usepackage{verbatim}
\usepackage{url}
\usepackage{hyperref}
\usepackage{enumitem}

\DeclareMathOperator{\cha}{char}

\DeclareMathOperator{\Gal}{Gal}

\DeclareMathOperator{\diag}{diag}

\DeclareMathOperator{\shHom}{\mathscr{H}\text{\kern -3pt {\calligra\large om}}\,}

\begin{document}
%
%   D e f i n i t i o n s
%
%
\theoremstyle{definition}
\newtheorem{Definition}{Definition}[section]
\newtheorem*{Definitionx}{Definition}
\newtheorem{Convention}[Definition]{Convention}
\newtheorem{Construction}{Construction}[section]
\newtheorem{Example}[Definition]{Example}
\newtheorem{Examples}[Definition]{Examples}
\newtheorem{Remark}[Definition]{Remark}
\newtheorem{Setup}[Definition]{Setup}
\newtheorem*{Remarkx}{Remark}
\newtheorem{Remarks}[Definition]{Remarks}
\newtheorem{Caution}[Definition]{Caution}
\newtheorem{Conjecture}[Definition]{Conjecture}
\newtheorem*{Conjecturex}{Conjecture}
\newtheorem{Question}[Definition]{Question}
\newtheorem{Questions}[Definition]{Questions}
\newtheorem*{Acknowledgements}{Acknowledgements}
\newtheorem*{Organization}{Organization}
\newtheorem*{Disclaimer}{Disclaimer}
\theoremstyle{plain}
\newtheorem{Theorem}[Definition]{Theorem}
\newtheorem*{Theoremx}{Theorem}
\newtheorem{Theoremy}{Theorem}
\newtheorem{Proposition}[Definition]{Proposition}
\newtheorem*{Propositionx}{Proposition}
\newtheorem{Lemma}[Definition]{Lemma}
\newtheorem{Corollary}[Definition]{Corollary}
\newtheorem*{Corollaryx}{Corollary}
\newtheorem{Fact}[Definition]{Fact}
\newtheorem{Facts}[Definition]{Facts}
\newtheoremstyle{voiditstyle}{3pt}{3pt}{\itshape}{\parindent}%
{\bfseries}{.}{ }{\thmnote{#3}}%
\theoremstyle{voiditstyle}
\newtheorem*{VoidItalic}{}
\newtheoremstyle{voidromstyle}{3pt}{3pt}{\rm}{\parindent}%
{\bfseries}{.}{ }{\thmnote{#3}}%
\theoremstyle{voidromstyle}
\newtheorem*{VoidRoman}{}

\numberwithin{equation}{section} 

\makeatletter
\@addtoreset{equation}{section}
\makeatother

% abgeschrieben aus The LaTeX Companion, 2nd edition,
% von Mittelback & Goossens
%
\newcommand{\prf}{\par\noindent{\sc Proof.}\quad}
\newcommand{\blowup}{\rule[-3mm]{0mm}{0mm}}
\newcommand{\cal}{\mathcal}
\newcommand{\Aff}{{\mathds{A}}}
\newcommand{\BB}{{\mathds{B}}}
\newcommand{\CC}{{\mathds{C}}}
\newcommand{\EE}{{\mathds{E}}}
\newcommand{\FF}{{\mathds{F}}}
\newcommand{\GG}{{\mathds{G}}}
\newcommand{\HH}{{\mathds{H}}}
\newcommand{\NN}{{\mathds{N}}}
\newcommand{\ZZ}{{\mathds{Z}}}
\newcommand{\PP}{{\mathds{P}}}
\newcommand{\QQ}{{\mathds{Q}}}
\newcommand{\RR}{{\mathds{R}}}
\newcommand{\Sph}{{\mathds{S}}}
\newcommand{\TT}{{\mathds{T}}}
\newcommand{\Liea}{{\mathfrak a}}
\newcommand{\Lieb}{{\mathfrak b}}
\newcommand{\Lieg}{{\mathfrak g}}
\newcommand{\Liem}{{\mathfrak m}}
\newcommand{\ideala}{{\mathfrak a}}
\newcommand{\idealb}{{\mathfrak b}}
\newcommand{\idealg}{{\mathfrak g}}
\newcommand{\idealj}{{\mathfrak j}}
\newcommand{\idealm}{{\mathfrak m}}
\newcommand{\idealn}{{\mathfrak n}}
\newcommand{\idealp}{{\mathfrak p}}
\newcommand{\idealq}{{\mathfrak q}}
\newcommand{\idealI}{{\cal I}}
\newcommand{\lin}{\sim}
\newcommand{\num}{\equiv}
\newcommand{\dual}{\ast}
\newcommand{\iso}{\cong}
\newcommand{\homeo}{\approx}
\newcommand{\mathds}[1]{{\mathbb #1}}
\newcommand{\mm}{{\mathfrak m}}
\newcommand{\pp}{{\mathfrak p}}
\newcommand{\qq}{{\mathfrak q}}
\newcommand{\rr}{{\mathfrak r}}
\newcommand{\pP}{{\mathfrak P}}
\newcommand{\qQ}{{\mathfrak Q}}
\newcommand{\rR}{{\mathfrak R}}
%
%  evtl. auch \"uber \mathbb oder \Bbb
%
\newcommand{\OO}{{\cal O}}
\newcommand{\calA}{{\cal A}}
\newcommand{\calD}{{\cal D}}
\newcommand{\calM}{{\cal M}}
\newcommand{\calO}{{\cal O}}
\newcommand{\calP}{{\cal P}}
\newcommand{\calT}{{\cal T}}
\newcommand{\calU}{{\cal U}}
\newcommand{\numero}{{n$^{\rm o}\:$}}
\newcommand{\mf}[1]{\mathfrak{#1}}
\newcommand{\mc}[1]{\mathcal{#1}}
\newcommand{\into}{{\hookrightarrow}}
\newcommand{\onto}{{\twoheadrightarrow}}
\newcommand{\Spec}{{\rm Spec}\:}
\newcommand{\BigSpec}{{\rm\bf Spec}\:}
\newcommand{\Spf}{{\rm Spf}\:}
\newcommand{\Proj}{{\rm Proj}\:}
\newcommand{\Pic}{{\rm Pic }}
\newcommand{\Picloc}{{\rm Picloc }}
\newcommand{\Br}{{\rm Br}}
\newcommand{\Cl}{{\rm Cl}}
\newcommand{\NS}{{\rm NS}}
\newcommand{\id}{{\rm id}}
\newcommand{\Sym}{{\mathfrak S}}
\newcommand{\Alt}{{\mathfrak A}}
\newcommand{\Aut}{{\rm Aut}}
\newcommand{\Inn}{{\rm Inn}}
\newcommand{\Out}{{\rm Out}}
\newcommand{\Hol}{{\rm Hol}}
\newcommand{\Autp}{{\rm Aut}^p}
\newcommand{\End}{{\rm End}}
\newcommand{\Hom}{{\rm Hom}}
\newcommand{\Ext}{{\rm Ext}}
\newcommand{\ord}{{\rm ord}}
\newcommand{\perf}{{\rm pf}}
\newcommand{\coker}{{\rm coker}\,}
\newcommand{\divisor}{{\rm div}}
\newcommand{\Def}{{\rm Def}}
\newcommand{\et}{{\rm \acute{e}t}}
\newcommand{\loc}{{\rm loc}}
\newcommand{\ab}{{\rm ab}}
\newcommand{\piet}{{\pi_1^{\rm \acute{e}t}}}
\newcommand{\pietloc}{{\pi_{\rm loc}^{\rm \acute{e}t}}}
\newcommand{\piN}{{\pi^{\rm N}_1}}
\newcommand{\piNloc}{{\pi_{\rm loc}^{\rm N}}}
\newcommand{\Het}[1]{{H_{\rm \acute{e}t}^{{#1}}}}
\newcommand{\Hfl}[1]{{H_{\rm fl}^{{#1}}}}
\newcommand{\Hcris}[1]{{H_{\rm cris}^{{#1}}}}
\newcommand{\HdR}[1]{{H_{\rm dR}^{{#1}}}}
\newcommand{\hdR}[1]{{h_{\rm dR}^{{#1}}}}
\newcommand{\Torsloc}{{\rm Tors}_{\rm loc}}
\newcommand{\defin}[1]{{\bf #1}}
\newcommand{\oX}{\cal{X}}
\newcommand{\oA}{\cal{A}}
\newcommand{\oY}{\cal{Y}}
\newcommand{\calC}{{\cal{C}}}
\newcommand{\calL}{{\cal{L}}}
\newcommand{\bmu}{\boldsymbol{\mu}}
\newcommand{\balpha}{\boldsymbol{\alpha}}
\newcommand{\bL}{{\mathbf{L}}}
\newcommand{\bM}{{\mathbf{M}}}
\newcommand{\bW}{{\mathbf{W}}}
\newcommand{\bD}{{\mathbf{D}}}%dihedral group
\newcommand{\bT}{{\mathbf{T}}}%tetrahedral group
\newcommand{\bO}{{\mathbf{O}}}%octahedral group
\newcommand{\bI}{{\mathbf{I}}}%icosahedral group
\newcommand{\AutScheme}{{\mathbf{Aut}}}
\newcommand{\InnScheme}{{\mathbf{Inn}}}
\newcommand{\OutScheme}{{\mathbf{Out}}}
\newcommand{\Norm}{{\mathbf{N}}}
\newcommand{\Cent}{{\mathbf{Z}}}
\newcommand{\Split}{{\mathbf{Split}}}
\newcommand{\BD}{{\mathbf{BD}}}
\newcommand{\BT}{{\mathbf{BT}}}
\newcommand{\BI}{{\mathbf{BI}}}
\newcommand{\BO}{{\mathbf{BO}}}
\newcommand{\C}{{\mathbf{C}}}
\newcommand{\Dic}{{\mathbf{Dic}}}
\newcommand{\SL}{{\mathbf{SL}}}
\newcommand{\PSL}{{\mathbf{PSL}}}
\newcommand{\PGL}{{\mathbf{PGL}}}
\newcommand{\MC}{{\mathbf{MC}}}
\newcommand{\GL}{{\mathbf{GL}}}
\newcommand{\Torus}{{\mathbf{T}}}
\newcommand{\Tors}{{\mathbf{Tors}}}

\newcommand{\bC}{{\mathbb{C}}}
\newcommand{\bR}{{\mathbb{R}}}

\newif\ifhascomments \hascommentstrue
\ifhascomments
  \newcommand{\matt}[1]{{\color{red}[[\ensuremath{\spadesuit\spadesuit\spadesuit} #1]]}}
  \newcommand{\christian}[1]{{\color{red}[[\ensuremath{\star\star\star} #1]]}}
\else
  \newcommand{\matt}[1]{}
  \newcommand{\christian}[1]{}
\fi

\makeatletter
\@namedef{subjclassname@2020}{\textup{2020} Mathematics Subject Classification}
\makeatother

\title[RDPs over nonclosed fields]{On rational double points over nonclosed fields}

\author{Christian Liedtke}
\address[CL]{TU M\"unchen, Zentrum Mathematik - M11, Boltzmannstr. 3, 85748 Garching bei M\"unchen, Germany}
\email{christian.liedtke@tum.de}

\author{Matthew Satriano}
\thanks{MS was partially supported by a Discovery Grant from the
  National Science and Engineering Research Council of Canada and a Mathematics Faculty Research Chair from the University of Waterloo.}
\address[MS]{Department of Pure Mathematics, University
  of Waterloo, Waterloo ON N2L3G1, Canada}
\email{msatrian@uwaterloo.ca}


\begin{abstract}
We compute the equations of all 
rational double point 
singularities and we determine their types
over perfect ground fields $k$ 
that arise as quotient singularities by finite 
linearly reductive subgroup schemes 
of $\SL_{2,k}$.
\end{abstract}
\maketitle
\setcounter{tocdepth}{1}

\section{Introduction}
In 1884, Felix Klein \cite{Klein} classified finite subgroups of $\SL_2(\CC)$ 
up to conjugacy and computed the resulting finite quotient singularities.
It turns out that there are two infinite series and three `exceptional' 
such groups and singularities.
The resulting singularities are precisely the 
rational double point singularities (RDP singularities for short),
which have shown up in many different contexts and 
which can be characterized in many different ways, see, for example,
\cite{Durfee}.

Over algebraically closed fields of characteristic $p>0$, 
one should work with finite group schemes rather than finite groups.
Moreover, one should require the group schemes to be 
linearly reductive,
which is automatic in characteristic zero.
The classification of finite linearly reductive
subgroup schemes of $\SL_{2,k}$ over an algebraically
closed field $k$ of characteristic $p>0$ is due to
Hashimoto \cite{Hashimoto}, but see also \cite[Section 4]{LS}.
If $p\neq\{2,3,5\}$, then every RDP singularity over an
algebraically closed field of characteristic $p$ 
arises as a quotient by a finite linearly reductive
subgroup scheme of $\SL_{2,k}$ - again, there
are two infinite series and three exceptional group schemes
and singularities.

This begs for the question what can be said about
RDP singularities over a field $k$ that is not algebraically
closed.
Let us recall that if $x\in X$ is an RDP singularity
over a field $k$, then there exists a 
minimal resolution of singularities $\pi:Y\to X$,
which is unique up to isomorphism.
The exceptional locus $\mathrm{Exc}(\pi)$
of $\pi$ is a union
of integral curves, whose dual resolution graph
$\Gamma=\Gamma(x\in X)$ is a finite Dynkin diagram,
called the \emph{type} of $x\in X$.
We will say that $x\in X$ is \emph{split}
or \emph{classical} if all integral curves
in $\mathrm{Exc}(\pi)$ are geometrically integral
or, equivalently, if 
$\Gamma(x\in X)=\Gamma(x_{\overline{k}},X_{\overline{k}})$.
It turns out that $x\in X$ is split if and only 
$\Gamma$ is simply-laced, that is, of type ADE.
We refer to \cite[\S24]{Lipman} for details.

It has long been known to the experts that not all types of 
RDP singularities
can be realized as quotients by finite subgroups of $\SL_2(k)$, 
even in characteristic zero and even over the real 
numbers $\RR$,
see, for example, \cite[Section 4.5]{BlumeMcKay} or
\cite[Chapter 6]{Slodowy}.

In this article, we will study finite linearly reductive
subgroup schemes of $\SL_{2,k}$ and their associated
quotient singularities in the case 
where $k$ is an arbitrary 
perfect field of characteristic $p\geq0$.

This turns out to be the correct generality to obtain
all types of RDP singularities over $k$.
The main result of this article is the following.

\begin{Theorem}\label{tm:twisted-invariants}
Let $k$ be a perfect field of characteristic $p\geq0$.
Let $G\subset\SL_{2,k}$ be a finite linearly reductive group scheme.
Then, $\Aff^2_k/G$ is given by one of the following equations:
\begin{center}
\begin{tabular}{ccll}
\phantom{XXX} & type & equation &   \\
\hline
\multicolumn{4}{l}{split forms} \\
\hline
& $A_{n-1}$ & $XY = Z^n$ & \\
& $D_{n+2}$ & $Y^2-X^2Z=Z^{n+1}$ &  $p\neq2$ \\
& $E_6$ & $X^2-4Y^4 = Z^3$ &  $p\neq2,3$\\
& $E_7$ & $X^2+4Y^3=YZ^3$&  $p\neq2,3$ \\
& $E_8$ & $X^2+Y^3+Z^5=0$ &  $p\neq2,3,5$ \\
\hline
\multicolumn{4}{l}{non-split forms of $A_{n-1}$} \\
\hline
& $B_{\beta(n)}$ & $dX^2-Y^2=4dZ^n$ &   $p\neq2$\\
& & $dX^2+XY+Y^2=Z^n$ & $p=2$ \\
\hline
\multicolumn{4}{l}{non-split forms of $D_{n+2}$ with $n\geq3$} \\
\hline
 & $C_{n+1}$ & $Y^2-X^2Z=4dZ^{n+1}$ &   $p\neq2$ \\
\hline
\multicolumn{4}{l}{non-split forms of $E_6$} \\
\hline
& $F_4$ & $X^2-4dY^4 = Z^3$ &  $p\neq2,3$ \\
\hline
\multicolumn{4}{l}{non-split forms of $D_4$} \\
\hline
 & $C_3$ & $Y^2-X^2Z=4dZ^3$ &   $p\neq2$ \\
%& $G_2$ & $Z^2=\alpha^{-1}X^3+\alpha Y^3$ & $p\neq2,3$\\
& $G_2$ & $2a Z^2 = bX(\Delta X^2-9Y^2)$ & $p\neq2,3$\\
&&\phantom{$12345 a Z^2 =$} $+Y(\Delta X^2-Y^2)$ & \\
%&& $2a Z^2 = b\Delta X^3 - 9bXY^2$ \\
%&&\phantom{$12345 a Z^2 =$} $-\Delta X^2Y + Y^3$ & \\
&& $Z^2 = bX^3 + aX^2Y+ Y^3$ & $p=3$
\end{tabular}
\end{center}
More precisely:
\begin{itemize}
\item If $\cha k\neq2$ and $d\in k^\times\setminus k^{\times2}$, then 
    $$
    d(X^2-4Z^n) \,=\,Y^2\,
    $$
    defines a $B_{\beta(n)}$-singularity.
    It splits over $k(\sqrt{d})$, where it becomes
    isomorphic to an $A_{n-1}$-singularity.
    Here, $\beta(n)$ is equal to $n/2$ if $n$ is even
    and equal to $(n-1)/2$ if $n$ is odd.
    \item If $\cha k\neq2$ and $d\in k^\times\setminus k^{\times2}$, then 
    $$
    Y^2\,=\,(X^2-4dZ^n)Z
    $$
    defines a $C_{n+1}$-singularity.
    It splits over $k(\sqrt{d})$, where it becomes 
    isomorphic to a $D_{n+2}$-singularity.
    \item If $\cha k\neq2$ and $d\in k^\times\setminus k^{\times2}$, then
    $$
    X^2-4dY^4 \,=\,Z^3
    $$
    defines a $F_4$-singularity.
    It splits over $k(\sqrt{d})$, where it becomes
    isomorphic to an $E_6$-singularity.
    \item Assume $p:=\cha k\neq2$ and let $L/k$ be the splitting field of
    an irreducible cubic $t^3+at+b$.
    Let $\Delta=-4a-27b^2$ be its discriminant. 
    Then,
    $$
    \begin{array}{ll}
    2aZ^2 \,=\, bX(\Delta X^2-9Y^2) + Y(\Delta X^2-Y^2) & \mbox{ if }p\neq3 \\
    Z^2 \,=\, bX^3 + aX^2Y+ Y^3 & \mbox{ if }p=3
    \end{array}
    $$
    defines a singularity of type $G_2$, see also
    \S\ref{subsec:S3twist} for simpler equations 
    in special cases.
    It splits over $L$, where it becomes isomorphic to a
    $D_4$-singularity.
\end{itemize}
\end{Theorem}

We would like to stress that this list is complete:
RDP singularities do exist also  in other characteristics than
covered in the above theorem, but then, they do not arise as
quotient singularities by linearly reductive group schemes,
see also \cite[Section 9]{RDP}.

\subsection{Descending the classical subgroups and their representations}
Let $k$ be a field and let $\overline{k}$ be 
an algebraic closure.
To prove Theorem \ref{tm:twisted-invariants}, our first step is to 
descend the finite subgroups of $\SL_{2,\overline{k}}$, which are 
well-known, to subgroup schemes over $k$.
Moreover, we would like to descend all representations over $\overline{k}$ to $k$ as well. 
We emphasize that \emph{even in characteristic zero}, it is not
\emph{a priori} clear how to do this. 
Here are two illustrative examples:
\begin{itemize}
\item If one wishes to descend all of the classical 
finite subgroups of $\SL_2(\CC)$ to a number field $k$, one's first inkling  is to simply consider the finite constant groups. 
This approach fails since we are not simply interested in descending the groups, 
but rather their \emph{embeddings} into $\SL_2$. 
For example, for constant groups of type E (namely the binary tetrahedral,
binary octahedral, and binary icosahedral group), 
the usual embeddings into $\SL_2(\CC)$ involve matrices
with entries $\sqrt{2}$ and $\sqrt{5}$. 
As a result, it is not clear how to descend the embeddings when $k$ does not contain such square roots.
\item The classical group of type $D_4$, 
namely $\BD_2(\CC)$ is isomorphic to $Q_8$ 
(the quaternion group of order 8) and it
embeds into $\SL_2(\CC)$ via its unique $2$-dimensional 
irreducible representation $\rho$. 
However, $\rho$ does not even descend to $\RR$ -- 
rather $\rho^{\oplus 2}$ descends, which is related
to Schur indices.
\end{itemize}
In both cases, one is forced to consider group schemes 
rather than constant groups.
The following result shows that the desired descent
is always possible.

\begin{Theorem}\label{thm:gp-sch-over-k-main-features}
Let $k$ be a field of characteristic $p\geq0$
and let $\overline{k}$ be an algebraic closure. 
Then, there exist finite linearly reductive subgroup schemes of $\SL_{2,k}$
\begin{equation}
    \label{eqn:ADE-over-perfect-k}
       \begin{array}{c|ccccc}
       &\bmu_n& \BD_n& \BT& \BO&\BI\\
       \hline
       p&&\neq2&\neq2,3&\neq2,3&\neq2,3,5
       \end{array},
\end{equation}
where the second row gives the restriction on $p$, and
that have the following properties:
\begin{enumerate}%[label=(\alph*)]
 \item\label{item-irrep} 
 If $G$ is from \eqref{eqn:ADE-over-perfect-k}, then every 
 representation of $G_{\overline{k}}$ descends to $k$.
 \item\label{item-conjugation} Up to conjugation, every finite 
 linearly reductive subgroup scheme of $\SL_{2,\overline{k}}$ 
 is of the form $G_{\overline{k}}$ for some 
 $G\subset\SL_{2,k}$ from \eqref{eqn:ADE-over-perfect-k}.
%   \item\label{item-irrep} Let $(G,\cha k)$ be as in (\ref{mainfeaturesstart})--(\ref{mainfeaturesend}). Then every representation of $G_{\overline{k}}$ descends to $G$.
\end{enumerate}
\end{Theorem}

Note that we do not make any uniqueness assumptions - 
we will show in a later article uniqueness up to conjugacy and 
up to inner twists.

The associated quotient singularities $\Aff^2/G$ are precisely
the classical forms of the RDP singularities.
In particular, all of them can be realized as 
quotient singularities
by finite linearly reductive subgroup schemes of $\SL_2$.
To the best of our knowledge, no one has previously written down 
such group schemes for type $E$, even in characteristic zero.

\subsection{McKay graphs}
Let $G$ be a finite linearly reductive group scheme over a 
field $k$ and let $\rho:G\to\GL_{n,k}$ be a representation.
In this article, this will always be an embedding $G\to\SL_{2,k}\subset\GL_{2,k}$.
We let $\{\psi_1,...,\psi_s\}$ be the irreducible
representations of $G$ over $k$ up to isomorphism. 
Associated to this data, the \emph{McKay graph}
$\widetilde{\Gamma}(G,\rho)$ is defined to be graph, whose 
\begin{itemize}
    \item vertices correspond to the $\{\psi_1,...,\psi_s\}$, and 
    \item there are $a_{ij}$ edges from $\psi_i$ to $\psi_j$,
    where
    $$
      a_{ij} \,:=\,\dim_k \Hom^G(\psi_i,\rho\otimes\psi_j).
    $$
\end{itemize}
We let $\Gamma(G,\rho)$ denote the graph that is obtained
by removing the vertex corresponding to the trivial representation
from the McKay graph $\widetilde{\Gamma}(G,\rho)$.
The McKay graph was originally defined by McKay \cite{McKay},
but see also \cite[Section 3]{LieMcKay} for the McKay graphs
of finite linearly reductive subgroup schemes of $\SL_{2,k}$.

If $G$ is from \eqref{eqn:ADE-over-perfect-k} in 
Theorem \ref{thm:gp-sch-over-k-main-features}
and if $\rho$ denotes the given embedding of $G$ into 
$\SL_{2,k}\subset\GL_{2,k}$, then
we have 
$$
\widetilde{\Gamma}(G,\rho)=\widetilde{\Gamma}
(G_{\overline{k}},\rho_{\overline{k}})
\mbox{ \quad and \quad }
\Gamma(G,\rho)=\Gamma(G_{\overline{k}},\rho_{\overline{k}})
$$
by Theorem \ref{thm:gp-sch-over-k-main-features}, Property (1).
The graph $\widetilde{\Gamma}(G,\rho)$ (resp. $\Gamma(G,\rho)$)
is an affine (resp. finite) Dynkin diagram of type 
$ADE$ and the associated quotient singularity
$x\in X=\Aff^2_k/G$ is an RDP
singularity of type $\Gamma(x\in X)$
equal to $\Gamma(G,\rho)$ - this is part of
the \emph{McKay correspondence} \cite{McKay}, but see 
also \cite{IN} and especially \cite{BlumeMcKay, LieMcKay}
in the context of nonclosed ground fields or
positive characteristic.

Moreover, still assuming that $G$ is as in
\eqref{eqn:ADE-over-perfect-k},
%Theorem \ref{thm:gp-sch-over-k-main-features}, 
we can
consider the normalizer subgroup scheme
$\Norm_{\SL_{2,k}}(G)$ in $\SL_{2,k}$.
The next theorem shows that these give naturally rise
to constant groups that capture the automorphism groups
of the graphs $\Gamma(G,\rho)$.

\begin{Theorem}\label{thm:NGmodGAutMcKay}
    Let $k$ be a field, 
    with an algebraic closure $\overline{k}$
    and let $G\subset\SL_{2,k}$ be as in 
    \eqref{eqn:ADE-over-perfect-k} of
    Theorem \ref{thm:gp-sch-over-k-main-features}. 
    Then, $\Norm_{\SL_{2,k}}(G)/G$ is a constant group scheme 
    and we have a natural isomorphism 
%    \begin{equation}\label{eqn:NmodG-McKay}
     $$
        \Norm_{\SL_{2,k}}(G)/G
        \quad\xrightarrow{\simeq}\quad
        \Aut\left(\Gamma(G_{\overline{k}},\rho_{\overline{k}})\right).
     $$ 
%    \end{equation}
\end{Theorem}

Theorem \ref{thm:NGmodGAutMcKay}, when combined with Theorem \ref{thm:NG-and-twists} below, 
plays a central role in the proof of 
Theorem \ref{tm:twisted-invariants}:
if $k$ is perfect, then it allow us to bound the degree of the Galois extension $L/k$
over which our twisted RDP singularities become classical forms
of RDP singularities. 
This limits the Galois extensions we need to consider 
when computing twists of the classical forms of the 
RDP singularities.

\subsection{Twisted subgroup schemes}
We end the introduction by mentioning a result,
which we need to establish Theorem \ref{thm:NGmodGAutMcKay} and
which seems to be known to the experts - we learnt it from a discussion
with Brian Conrad, we do not claim much originality, but we were unable to find 
it in the literature.
Namely, we show that twists of a group scheme $G$ in a group scheme $X$ 
are classified in terms of $\Norm_G(X)$-torsors.

\begin{Theorem}\label{thm:NG-and-twists}
    Let $k$ be a field, 
    let $k\subseteq K$ be a field extension,
    let $X$ be an affine and
    finitely presented group scheme over $k$, let $G\subset X$ 
    be a finitely presented closed subgroup scheme over $k$, 
    and let $\Norm_X(G)$ 
    be the normalizer subgroup scheme of $G\subset X$.
    Then, there exists natural bijections of the following sets: 
    \begin{enumerate}
    \item $X(k)$-conjugacy classes of closed finitely presented $k$-subgroup schemes $X$
    of $X$ such that $H_K$ is conjugate to $G_K$ over $X_K$.
    \item Isomorphism classes of left $\Norm_X(G)$-torsors 
    over $k$ that admit an equivariant closed immersion 
    into $X$ and that split over $K$.
%        \item Isomorphism classes of left $\Norm_X(G)$-torsors
%        over $k$, whose pushout to $X$ are trivial over $K$.
    \end{enumerate}
\end{Theorem}

In particular, if $k$ is a perfect field and
$G\subset\SL_{2,k}$ is as in \eqref{eqn:ADE-over-perfect-k}
of Theorem \ref{thm:gp-sch-over-k-main-features},
then the previous theorem shows that finite linearly
reductive subgroup schemes $G'\subset\SL_{2,k}$ that
become conjugate to $G\subseteq\SL_{2,k}$ over $\overline{k}$
are classified by the pointed Galois cohomology set
$$
H^1\left(\Gal(\overline{k}/k),\,\Norm_{\SL_{2,k}}(G)\right),
$$
see Corollary \ref{cor: brian} and Remark \ref{rem: brian}.

\begin{Remark}
\label{rem: classification}
This yields a cohomological description of
\emph{all} finite linearly reductive subgroup schemes
of $\SL_{2,k}$ by Theorem \ref{thm:gp-sch-over-k-main-features}, Property (2).
\end{Remark}

\subsection{Outlook}
In work in preparation, we will give a complete and detailed 
description of finite linearly reductive subgroup schemes
of $\SL_{2,k}$ and the associated quotient singularities
over arbitrary ground fields.
This will be more of a monography.

\begin{VoidRoman}[Acknowledgements]
 We thank Jason Bell, Dori Bejleri, Colin Ingalls, 
 Brett Nasserden, and Paul Pollack for helpful conversations. 
 We especially thank Brian Conrad for his help 
 in proving Theorem \ref{thm:NG-and-twists}.
\end{VoidRoman}

\tableofcontents

\section{Classifying twisted forms using the normalizer:~Theorem \ref{thm:NG-and-twists}}

In this section, we prove Theorem \ref{thm:NG-and-twists}. 
Recall the normalizer subgroup scheme 
of a group scheme $G$ in a group scheme $X$ 
is given by the functor
$$
\Norm_X(G)(S) \,:\,=
\left\{x\in X(S) \,\mid\, \forall T\xrightarrow{f} S,\, g\in G(T),\, 
f^*(x)g f^*(x)^{-1}\in G(T)
\right\}.
$$
We refer to \cite[Section 2.1]{Conrad} for details, including representability
and general properties of $\Norm_X(G)$.
In particular, if $G$ is a closed subgroup scheme of
an affine group scheme $X$ over $k$, then
$\Norm_X(G)$ will be a closed subgroup scheme of $X$
and thus, it is an affine group scheme over $k$.

\begin{proof}[{Proof of Theorem \ref{thm:NG-and-twists}}]

$(1)\Rightarrow(2)$: 
Let $H$ be a finitely presented closed $k$-subgroup scheme of $X$. Then by \cite[Proposition 2.1.2]{Conrad}, the functor %from schemes over $k$ to sets
$$
S\,\mapsto\, \{x\in X(S)| x\cdot H_S\cdot x^{-1}=G_S\}.
$$
is representable by a closed subscheme $E_H$ of $X$. %as it is the intersecton of two transporter schemes or b/c it's an NG torsor
It is also a left $\Norm_X(G)$-torsor and since we have
$E_H(K)\neq\emptyset$ by our assumptions, this $\Norm_X(G)$-torsor  splits over $K$.
Furthermore, the embedding $E_H\subset X$ 
is $\Norm_X(G)$-equivariant.

$(2)\Rightarrow(1)$: 
Let $E$ be a left $\Norm_X(G)$-torsor over $k$ that splits over $K$.
Since $E(K)\neq\emptyset$, we may pick an element
$t\in E(K)$.
Next, we note that $E(K)=\Norm_X(G)\cdot K$, which shows that
$t^{-1}G_Kt\subset X_K$ is independent of the choice of element
$t\in E(K)$.
Then, we define
$$
H_E \,:=\, %G\wedge E \,=\, 
(G\times_kE)/\Norm_X(G) \,\to\, 
(X\times_kE)/\Norm_X(G)\,=\,X,
$$
where the action of $\Norm_X(G)$ on $G\times_k E$ and $X\times_kE$
is diagonal via conjugation on the first factor and the 
torsor structure on the second factor.
This way, we obtain a subgroup scheme $H_E\subseteq X$ over $k$ 
that becomes $t^{-1}G_Kt\subseteq X_K$ over $K$.

%It is easy to check that the associated transporter scheme $E_H\to\Spec k$ as defined in the first paragraph of the proof is an $\Norm_X(G)$-torsor. It splits over $K$ and it comes with an equivariant closed immersion into $X$.

It is straightforward to check that the constructions
$H\mapsto E_H$ and $E\mapsto H_E$ are inverse to each other.
\end{proof}

If $k$ is a perfect field with algebraic closure $\overline{k}$, then
we denote its absolute Galois group by $\Gal_k:=\Gal(\overline{k}/k)$.

\begin{Corollary}
 \label{cor: brian}
    We keep the assumptions and notations of Theorem \ref{thm:NG-and-twists}.
    Assume moreover that the field $k$ is perfect. 
    Then, there exists a natural bijection of pointed 
    sets between
    $$
    \ker\left(  
    H^1(\Gal_k,\, \Norm_X(G)(\overline{k})) \,\to\,
    H^1(\Gal_k,\, X(\overline{k}))\right)
    $$
%   \Hfl{1}(\Spec k,\Norm_X(G)) \,\to\, \Hfl{1}(\Spec k, X) \right)
    and $X(k)$-conjugacy classes of closed $k$-subgroup schemes $H$
    of $X$ such that $H_{\bar{k}}$ is conjugate to $G_{\bar{k}}$ over 
    $X_{\bar{k}}$.
\end{Corollary}

\begin{proof}
Let $k\subseteq K$ be a finite Galois extension.
We note that a left $\Norm_X(G)$-torsor $E$ over $k$
that splits over $K$
admits an equivariant closed immersion 
into $X$ if and only if its pushout to $X$ is trivial
and thus, if and only if it lies in the kernel
of the map
$H^1(\Gal(K/k),\Norm_X(G)(K))\to H^1(\Gal(K/k),X(K))$.

Passing to the limit over all finite subextensions of $k$ 
inside $\bar{k}$, 
we obtain a bijection between the set of isomorphism classes of 
$\Norm_X(G)$-torsors over $k$ that admit an equivariant closed 
immersion into X (as all such split over some finite extension of $k$)
and the set of $X(k)$-conjugacy classes of closed 
$k$-subgroup schemes $H$ of $X$ such 
that is $H_{\bar{k}}$ is conjugate to $G_{\bar{k}}$
over $X_{\bar{k}}$.
\end{proof}

\begin{Remark}
\label{rem: brian}
 This corollary becomes particularly useful in the case where 
 $H^1(\Gal_k,X(\overline{k}))$ is trivial.
 We note that there is a vast literature on Galois cohomology
 of linear algebraic groups, see, for example,
 \cite{SerreGaloisCohomology}.
 For example, it is well-known that we have
 $$
 H^1\left(\Gal_k, \GL_{n}(\overline{k})\right) \,=\,\{\ast\}
 \mbox{ \qquad for all }n\geq1,
 $$ 
 see, for example, \cite[Chapter III.\S1, Lemma 1]{SerreGaloisCohomology}.
% (the case $n=1$ is the famous `Hilbert 90 Theorem').
 Taking Galois cohomology
 in the short exact sequence
 $$
 1 \,\to\,\SL_{n,k}\,\to\,\GL_{n,k}\,
 \stackrel{\det}{\longrightarrow} 
 \,\GG_{m,k}\,\to\,1,
 $$
 we find
 $H^1(\Gal_k, \SL_{n}(\overline{k})) =\{\ast\}$
 for all $n\geq2$, which we will need below.
\end{Remark}

\section{Twisted forms of classical groups:~Theorem \ref{thm:gp-sch-over-k-main-features}}

In this section, we prove Theorem \ref{thm:gp-sch-over-k-main-features}. 
In \S\ref{subsec:def-our-gp-schs}, we construct the group schemes over $k$ 
that appear in Theorem \ref{thm:gp-sch-over-k-main-features}, and in
\S\ref{subsec:irreps-decsend}, we prove that the irreducible representations of 
the classical groups descend to irreducible representations of our 
group schemes over $k$.

\subsection{Definitions and properties of our group schemes over $k$}\label{subsec:def-our-gp-schs}

We first give the definitions of the group schemes that appear in 
Theorem \ref{thm:gp-sch-over-k-main-features}. 
The definitions of $\BT$, $\BO$, and $\BI$ require justification which is given 
after statement of the definitions.
Let $k$ be a field.

\begin{enumerate}
    \item \textbf{Roots of unity.} Recall the definition of the group scheme
    $$
    \bmu_n \,:=\, \Spec k[u]/(u^n-1).
    $$
    We have a closed embedding
    $$
    \rho\colon\bmu_n\to\SL_{2,k},\quad u\mapsto 
    \begin{bmatrix}
    u&0\\
    0&u^{-1}
    \end{bmatrix}.
    $$

\item\textbf{Binary dihedral group scheme.}  
Assume $\cha k\neq2$.
In a previous paper \cite[Section 4]{LS}, we considered the group scheme
$$
\BD_n^* \,=\,\Spec k[a,b]/(ab,a^{2n}+b^{2n}-1).
$$
We have a closed embedding
$$
 \rho\,\colon\,\BD_n^*\,\to\,\SL_{2,k} \,=\,
 \Spec k[a,b,c,d]/(ad-bc-1)
$$
defined by the ideal
 $$
c+b^{2n-1},\quad d-a^{2n-1},\quad ab.
$$
We define the binary dihedral group scheme $\BD_n$ to be the 
following twist of $\BD_n^*$. 
If $i:=\sqrt{-1}\in k$, we define $\BD_n=\BD_n^*$. 
If $i\notin k$, we define $\BD_n$ as follows. 
Note that the diagonal matrix $\diag(i,-i)$ normalizes 
$\BD_{n,k(i)}^*$, hence letting $\Gal(k(i)/k)$ act on $\BD_{n,k(i)}^*$ 
via conjugation by this matrix, 
Theorem \ref{thm:NG-and-twists} yields a closed embedding
$$
\BD_n \,\to\, \SL_{2,k}.
$$
This subgroup scheme is conjugate to $\BD_n^*$ over $k(i)$. 
Explicitly, conjugation by $\diag(i,-i)$ acts as 
$a,b\mapsto a,-b$ and taking the invariant ring under $\Gal(k(i)/k)$ shows
$$
\BD_n \,\cong\, \Spec k[a,b]/(ab,a^{2n}+(-1)^nb^{2n}-1).
$$
In particular, we have $\BD_{2n}^*=\BD_{2n}$ as subgroup schemes of $\SL_{2,k}$.

\item\textbf{Binary octahedral group scheme.} 
Assume $\cha k\neq2$.
We define the binary octahedral group scheme to be the normalizer
$$
\BO \,:=\,\Norm_{\SL_{2,k}}(\BD_2).
$$

\item\textbf{Binary tetrahedral group scheme.} 
We next define the binary tetrahedral group scheme $\BT$ whenever $\cha k\neq2,3$; however, we first define an auxiliary group scheme $\BT^*$ whenever $\cha k\neq2$. Proposition \ref{prop:BD2-1d-irreps} below will show that the action of 
$\BO$ on $\BD_2$ permutes three specific representations 
$\BD_2\to\bmu_2\subset\GL_{1,k}$. 
Hence, we obtain maps
$$
\BO \,\to\, \Sym_3 \,\xrightarrow{\textrm{sign}}\,\ZZ/2,   
$$
where $\Sym_3$ denotes the symmetric group on 3 letters.
We define $\BT^*$ to be the kernel of this composition. 

Now assume $\cha k\neq2,3$. If $k$ contains all $3$.rd roots of unity, we define $\BT=\BT^*$. Otherwise, let $\zeta_{24}\in\overline{k}$ be a primitive $24$.th root of unity; note $\Gal(k(\zeta_{24})/k)\simeq\Gal(k(\zeta_{24})/k(\zeta_3))\times\Gal(k(\zeta_{24})/k(\zeta_8))$. Consider the Galois action on $\BT^*_{k(\zeta_{24})}$,
where $\Gal(k(\zeta_{24})/k(\zeta_3))$ acts trivially and the generator of $\Gal(k(\zeta_{24})/k(\zeta_8))$ acts via conjugation 
by the matrix 
$$\begin{bmatrix}
        0 & \zeta_8\\
        -\zeta_8^{-1} & 0
    \end{bmatrix} \,\in\, \SL_2\left(k(\zeta_{24})\right).
$$    
Proposition \ref{prop:BD2-1d-irreps} will show that this matrix 
normalizes $\BT^*_{k(\zeta_{24})}$. 
Letting $\BT$ be the quotient $\BT^*_{k(\zeta_{24})}$ by $\Gal(k(\zeta_{24})/k)$, 
Theorem \ref{thm:NG-and-twists} tells us $\BT$ is a subgroup scheme 
of $\SL_{2,k}$ and is conjugate to $\BT^*$ over $k(\zeta_{24})$.

\item\textbf{Binary icosahedral group scheme.} 
We next define the binary tetrahedral group scheme whenever 
$\cha k\neq\{2,3,5\}$. 
The classical binary icosahedral group has a natural faithful 
$2$-dimensional irreducible representation $\widetilde{\rho}$ to 
$\SL_{2,k(\zeta_5)}$, which we describe later. 
There is an outer automorphism $\tau$ of the binary icosahedral group 
and we prove in Proposition \ref{prop:BI-descends} 
that if $\eta$ is a generator of $\Gal(k(\zeta_5)/k)$, then 
$$
    \eta\widetilde{\rho} \,=\,\widetilde{\rho}\tau^m,
$$
where $m$ depends on whether $\zeta_5$ or $\zeta_5+\zeta_5^{-1}$ is in $k$. 
Thus, we obtain a Galois action on the classical binary icosahedral group 
where $\sigma$ acts via $\tau^m$ and this action is compatible with 
the embedding $\widetilde{\rho}$. 
Hence, $\widetilde{\rho}$ descends to a faithful representation
$$
\rho \,\colon\,\BI\,\to\,\SL_{2,k}
$$
of a group scheme which we call the binary icosahedral group scheme $\BI$.
\end{enumerate}

We begin by proving some properties of $\BO$.

\begin{Lemma}\label{l:BO-etale}
    If $k$ is a field with $\cha k\neq2$, then $\BD_2$ and $\BO$ are \'etale over $k$. Furthermore, $\BO_{\overline{k}}$ is the classical binary octahedral group.
\end{Lemma}

\begin{proof}
We can show $\BD_2$ directly from the definition by proving $\Omega^1_{\BD_2/k}=0$. 
To see this, we note $adb+bda = d(ab)=0$ and so, $b^4da = 0$. 
Also, $a^3da + b^3db = 0$, so $a^4da = 0$ and as
a result, $da = (a^4+b^4)da = 0$. 
Similarly, we find $db=0$.

To check $\BO$ is \'etale, by fpqc descent, it suffices to 
base change to $\overline{k}$, so we may assume $k$ is 
algebraically closed. 
Then we see from the definition of $\BD_2$ that it is isomorphic to the quaternion group $Q_8$ embedded into $\SL_2$ 
via its unique irreducible $2$-dimensional representation $\psi$. 
So, we are reduced to proving $\BO:=\Norm_{\SL_2}(Q_8)$ over the algebraically closed field $k$ is the constant group given by the classical binary octahedral group. Let $R$ be a $k$-algebra and let $A\in\BO(R)$. For any $C\in\SL_2(k)$, we let $C_R\in\SL_2(R)$ be the induced $R$-point. Let $\zeta_8\in\overline{k}$ be a primitive $8$.th root 
of unity. 
Since conjugation by 
$$
B:=\frac{1}{\sqrt{2}}
\begin{bmatrix}
    \zeta_8 &\zeta_8\\
    -\zeta_8^{-1} &\zeta_8^{-1}
\end{bmatrix}\in\BO(\overline{k})
$$
cyclically permutes $i,j,k\in Q_8$, multiplying $A$ by an appropriate power of $B_R$, we may assume $A\cdot \psi(i)_R\cdot A^{-1}=\psi(\pm i)_R$. 
Conjugation by $j$ sends $i$ to $-i$, so multiplying $A$ by $\psi(j)_R$ if necessary, we may additionally 
assume $A\cdot\psi(i)_R\cdot A^{-1}=\psi(i)_R$. 
Since $\psi(i)_R=\diag(i,-i)$, a direct computation 
shows $A=\diag(\lambda,\lambda^{-1})\in\GG_m(R)$. 
Multiplying $A$ by $\psi(i)_R$, we can assume 
$A\cdot\psi(j)_R\cdot A^{-1}\in\{\psi(j)_R,\psi(k)_R\}$, and so $\lambda=\pm1$ or $\lambda=\pm\zeta_8$. Thus, $A$ is induced by a $k$-point and lives in the classical binary octahedral group.
\end{proof}

\begin{Lemma}\label{l:BO-basics}
Let $k$ be a field with $\cha k\neq2$, let $S$ be a $k$-scheme, 
and let $A:=\begin{bmatrix}
    a & b\\ c&d
\end{bmatrix}\in\BO(S)$. 
Then, one of the following holds:
\begin{enumerate}
    \item\label{l:BO-basics::ad-1} $ad=1$, $b=c=0$, and $a^8=1$.
    \item\label{l:BO-basics::bc-1} $bc=-1$, $a=d=0$, and $b^8=1$.
    \item\label{l:BO-basics::nonBD4} $ad=-bc=\frac{1}{2}$, $a^4=b^4=:\frac{\epsilon}{4}$, $(ab)^2=:\frac{\eta}{4}$, $(ac)^2=(bc)^2=\frac{\eta\epsilon}{4}$, where $\epsilon,\eta=\pm1$.
\end{enumerate}
\end{Lemma}

\begin{proof}
We may assume that $S$ is connected. 
Since $\BO$ is \'etale over $k$ by Lemma \ref{l:BO-etale}, 
it is the disjoint union of spectra of finite separable field extensions 
of $k$.
Hence, our $S$-valued point is induced by an $L$-valued point for a finite separable field extension $L/k$, so it suffices to assume $S=\Spec L$. 
Then by fpqc descent, 
to verify the lemma, it suffices to replace $L$ by an algebraic closure, 
so we may assume $k=\overline{k}$ and $S=\Spec k$. 
In this case, we know that $\BO$ is the classical binary octahedral group,
whose points are given by $\BD_4$ 
(corresponding to cases (\ref{l:BO-basics::ad-1}) and (\ref{l:BO-basics::bc-1})) or 
$$
\frac{1}{\sqrt{2}}
\begin{bmatrix}
    \zeta^i &\zeta^j\\
    -\zeta^{-j} &\zeta^{-i}
\end{bmatrix}
$$
where $\zeta\in\bmu_8(k)$ and $i\equiv j\pmod{4}$ 
(corresponding to case (\ref{l:BO-basics::nonBD4})).
\end{proof}

The next result describes the $\BO$-action on $1$-dimensional representations of $\BD_2$. 
This yields the homomorphism $\BO\to\Sym_3$ asserted above, 
which is used to define $\BT^*$ above.

\begin{Proposition}\label{prop:BD2-1d-irreps}
    The maps 
    $$
     \BD_2=\Spec k[\alpha,\beta]/(\alpha\beta,\alpha^4+\beta^4-1)\,\to\,\bmu_2 \,\subset\,\GL_{1,k}
    $$ 
    given by
    $$
        \psi^+(\alpha,\beta) \,=\, \alpha^2+\beta^2,\quad 
        \psi^-(\alpha,\beta) \,=\,\alpha^2-\beta^2,\quad 
        \psi^0(\alpha,\beta) \,=\, \alpha^4-\beta^4
    $$
    are representations. 
    If $k$ is a field with $\cha k\neq2$, then the action of $\BO$ on $\BD_2$ 
    permutes $\psi^+,\psi^-,\psi^0$.
    More precisely, let 
    $$
     A \,=\,\begin{bmatrix}
        a&b\\ c&d
    \end{bmatrix}\,\in\,\BO.
    $$
    Then,
    \begin{enumerate}
        \item[(i)] if $A\in\BD_2$, then it fixes $\psi^\pm$ and $\psi^0$.
    
        \item[(ii)]\label{surj-map-to-z2} if $A$ is not in $\BD_2$ and is in cases (\ref{l:BO-basics::ad-1}) or (\ref{l:BO-basics::bc-1}) of Lemma \ref{l:BO-basics}, then it fixes $\psi^0$ and swaps $\psi^\pm$.

        \item[(iii)] if $A$ is case (\ref{l:BO-basics::nonBD4}) of Lemma \ref{l:BO-basics} with $\eta=1$, then $\psi^+,\psi^-,\psi^0\mapsto \psi^\epsilon,\psi^0,\psi^{-\epsilon}$.

        \item[(iv)] if $A$ is case (\ref{l:BO-basics::nonBD4}) of Lemma \ref{l:BO-basics} with $\eta=-1$, then $\psi^+,\psi^-,\psi^0\mapsto \psi^0,\psi^{-\epsilon},\psi^\epsilon$.
    \end{enumerate}
\end{Proposition}
\begin{proof}
First we note that $\psi^0=\psi^+\psi^-=\psi^-\psi^+$. 
We also observe that if $(\alpha,\beta)\in\BD_2$, then multiplying the 
equation $\alpha^4+\beta^4=1$ by $\alpha^i$ or $\beta^i$, we have $\alpha^{4+i}=\alpha^i$ and $\beta^{4+i}=\beta^i$ for $i\geq0$.

Let us first justify why $\psi^\pm,\psi^0$ are homomorphisms. 
It suffices to verify this for $\psi^\pm$ by our expression for $\psi^0$ above. 
We see
\begin{align*}
\psi^\pm((\alpha,\beta)(\gamma,\delta)) &=    \psi^\pm(\alpha\gamma-\beta\delta^3,\alpha\delta+\beta\gamma^3)=(\alpha\gamma-\beta\delta^3)^2\pm(\alpha\delta+\beta\gamma^3)^2\\
&=(\alpha\gamma)^2+(\beta\delta^3)^2\pm(\alpha\delta)^2\pm(\beta\gamma^3)^2\\
&=(\alpha^2\pm \beta^2)(\gamma^2\pm \delta^2).
\end{align*}

Next, let us prove the images of $\psi^\pm,\psi^0$ lie in $\bmu_2$. 
Again, we only need to verify this for $\psi^\pm$. 
We find
$$
\psi^\pm(\alpha,\beta)^2\,=\,(\alpha^2\pm \beta^2)^2\,=\,\alpha^4+\beta^4\,=\,1.
$$

We turn now to the final claim of the proposition. 
Since $\psi^0=\psi^+\psi^-$, it suffices to verify the claim for
$\psi^\pm$. 
Let $(\alpha,\beta)\in\BD_2$ and $A=\begin{bmatrix}
        a&b\\ c&d
    \end{bmatrix}\in\BO$. 
If $ad=1$, $b=c=0$, and $a^8=1$, then $A$ sends $(\alpha,\beta)$ 
to $(\alpha,a^2\beta)$. Note
$$
    \psi^\pm(\alpha,a^2\beta)\,=\,\alpha^4\pm a^4\beta^4,
$$
which is $\psi^\pm(\alpha,\beta)$ if $a^4=1$
(that is, $A\in\BD_2$) and $\psi^\mp(\alpha,\beta)$ 
if $a^4=-1$ (that is, $A\notin\BD_2$). 
If on the other hand,  $bc=-1$, $a=d=0$, and $b^4=1$,
then $A$ sends $(\alpha,\beta)$ to $(\alpha^3,b^2\beta^3)$ 
and the statement in the proposition is verified analogously.

Finally, if $A$ is in case (\ref{l:BO-basics::nonBD4}) 
of Lemma \ref{l:BO-basics}, then $A$ maps
$$
\alpha^2\,\mapsto\, 
\frac{1}{4}(\alpha+\alpha^3)^2+((ac)^2+(bd)^2)\beta^2+2abcd\beta^4
\,=\,
\frac{1}{2}(\psi^0+\psi^{\eta\epsilon})(\alpha,\beta)
$$ 
and  
$$
\beta^2\,\mapsto\,
(a^4+b^4)\beta^2+2(ab)^2\beta^4+2(ab)^2(\alpha^2-\alpha^4)
\,=\,
\frac{\eta}{2}(\psi^{\eta\epsilon}-\psi^0)(\alpha,\beta),
$$
from which the proposition follows.
\end{proof}

As explained at the beginning of this section, the above results establish the 
existence of $\BT^*$ and hence, $\BT$. 
We now prove the following property of $\BT$.

\begin{Proposition}\label{prop:def-BT}
    Let $k$ be a field with $\cha k\neq2,3$ and let 
    $\zeta_{24}\in\overline{k}$ be a primitive $24$.th root of unity. 
    The kernel of the map 
    $\BO_{k(\zeta_{24})}\to \Sym_3\xrightarrow{\textrm{sign}}\ZZ/2$ 
    yields $\BT^*_{k(\zeta_{24})}\to\ZZ/3$, which descends to
    $$
    \BT\to\bmu_3.
    $$
\end{Proposition}

\begin{proof}
Let $\zeta_8=\zeta_{24}^3$ and let $\zeta_3=\zeta_{24}^8$. 
By Proposition \ref{prop:BD2-1d-irreps}, we have an exact sequence
$$
 1\,\to\,\BD_2\,\to\,\BO\,\to\, \Sym_3 \,\to\, 1,
$$
which therefore yields an exact sequence
\begin{equation}\label{eqn:BD2-BTstar}
    1\,\to\,\BD_2\,\to\,\BT^*\,\xrightarrow{\pi}\,\ZZ/3\,\to\, 1.
\end{equation}
Consider the matrix 
$A:=\frac{1}{\sqrt{2}}\begin{bmatrix}
        \zeta_8 & \zeta_8\\
        -\zeta_8^{-1} & \zeta_8^{-1}
    \end{bmatrix}\in \SL_2(k(\zeta_{24}))$
and note that $\sqrt{2}=\zeta_8+\zeta_8^{-1}\in k(\zeta_{24})$. 
Proposition \ref{prop:BD2-1d-irreps} 
shows $A$ acts as an even permutation, hence $A\in\BT^*(k(\zeta_{24}))$ 
and maps under $\pi$ to a generator of $\ZZ/3$. 
The $\Gal(k(\zeta_{24})/k)$-action used to define $\BT$ factors through $\Gal(k(\zeta_3)/k)$ and the generator of $\Gal(k(\zeta_3)/k)$ acts as conjugatiion by the matrix 
$B:=\begin{bmatrix}
        0 & \zeta_8\\
        -\zeta_8^{-1} & 0
    \end{bmatrix}$. 
Note that $B$ normalizes $\BD_2$, hence the short exact 
sequence \eqref{eqn:BD2-BTstar} is $\Gal(k(\zeta_{24})/k)$-equivariant. 
Since $BAB^{-1}=A^{-1}$, we see $\Gal(k(\zeta_{24})/k)$ 
acts non-trivially on $\ZZ/3$; specifically, $\Gal(k(\zeta_{24})/k)$ acts through the quotient $\Gal(k(\zeta_{3})/k)$ and the generator of $\Gal(k(\zeta_{3})/k)$ maps to the unique non-trivial element of $\Aut(\ZZ/3)$. Hence the resulting Galois twist of $\ZZ/3$ is given by $\mu_3$, and so the map $\BT^*\to\ZZ/3$ descends to $\BT\to\bmu_3$.
\end{proof}

Finally, the following proposition is used in in the definition of $\BI$ above.

\begin{Proposition}\label{prop:BI-descends}
Let $k$ be a field with $\cha k\neq\{2,3,5\}$ 
and $\zeta_5\in\overline{k}$ be a primitive $5$.th root of unity. 
\begin{enumerate}
\item
The classical binary icosahedral constant group can be 
presented as
$$
\left\langle
r,\,t\,\mid\,
  r^2\,=\,t^5\,=\,(rt^{-1})^3,\, (r^2)^2=1
\right\rangle \,.
$$
The element $r^2$ is usually denoted $-1$ and it generates
the center.
There is a faithful $2$-dimensional representation to 
$\SL_{2,k(\zeta_5)}$ given by
$$
\widetilde{\rho}\quad \colon\quad 
t\,\mapsto\,\begin{bmatrix} -\zeta_5&0\\ 0 &-\zeta_5^{-1}\end{bmatrix},\quad
r\,\mapsto\,
\frac{1}{\sqrt{5}}\begin{bmatrix}
        \zeta_5^{-2}-\zeta_5^2 & -\zeta_5^{-1}+\zeta_5\\ 
        -\zeta_5^{-1}+\zeta_5 & -\zeta_5^{-2}+\zeta_5^2
    \end{bmatrix}.
$$
There exists a non-trivial outer automorphism $\tau$ 
of the classical binary icosahedral 
group acting via
$$
\tau\quad\colon\quad t\,\mapsto\,t^3,\quad 
    rtr^{-1}\,\mapsto\,(rtr^{-1})^{-3}.
$$
\item
Let $\sigma(\zeta_5)=\zeta_5^3$. 
Then 
$$
    \sigma\widetilde{\rho} \,=\,\widetilde{\rho}\tau.
$$
Since $\Gal(k(\zeta_5)/k)$ is generated by a power of $\sigma$ 
(depending on whether $\zeta_5$ or $\zeta_5+\zeta_5^{-1}$ lies in $k$),
we see $\widetilde{\rho}$ descends to a faithful representation of 
a group scheme $\rho\colon\BI\to\SL_{2,k}$.
\end{enumerate}
\end{Proposition}

\begin{proof}
The fact that $\widetilde{\rho}$ is a faithful $2$-dimensional representation 
and that $\tau$ is an outer automorphism are standard facts about the classical binary icosahedral group.
Note that $t$ and $rtr^{-1}$ generate the binary icosahedral group since $(rt)^3=r^2$, 
so $r=trtrt=-t(rtr^{-1})t=t^6(rtr^{-1})t$. 
A straightforward computation shows 
$\sigma\widetilde{\rho}(t)=\widetilde{\rho}(t^3)$ and 
$\sigma\widetilde{\rho}(rtr^{-1})=\widetilde{\rho}(rt^{-3}r^{-1})$, so 
$\sigma\widetilde{\rho}=\widetilde{\rho}\tau$. 
Lastly, $\Gal(k(\zeta_5)/k)$ is generated by $\sigma$ 
(resp.~$\sigma^2$, resp.~$\sigma^4=1$) 
if $\zeta_5+\zeta_5^{-1}\notin k$ 
(resp.~$\zeta_5\notin k$ but $\zeta_5+\zeta_5^{-1}\in k$, resp.~$\zeta_5\in k$).
\end{proof}

This establishes the existence of the group schemes $\bmu_n$, $\BD_n$,
$\BT$, $\BO$, and $\BI$, we turn to our first main theorem.

\begin{proof}[Proof of Theorem \ref{thm:gp-sch-over-k-main-features}: existence
and assertion (\ref{item-conjugation})]
In \cite[Theorem 3.8]{Hashimoto}, Hashimoto classifies finite 
linearly reductive subgroup schemes of $\SL_{2,\overline{k}}$ up to conjugacy, 
so we need only show that our group schemes \eqref{eqn:ADE-over-perfect-k} 
agree with his classification over $\overline{k}$ up to conjugation.
We note that this implies linear reductivity of our group schemes over $k$ by \cite[Propostion 2.4(b)]{AOV}. 
In fact, since $\BD_n^*$ (resp.~$\BT^*$) agrees with $\BD_n$ (resp.~$\BT$) 
up to conjugation over a finite extension of $k$, it is good enough to show 
our result with $\BD_n^*$ and $\BT^*$ replaced by $\BD_n$ and $\BT$.

The subgroup scheme $(A_n)$ in \cite[Theorem 3.8]{Hashimoto} is equal to
$\bmu_{n,\overline{k}}$. 
Next, conjugating $\BD^*_{n,\overline{k}}$ by 
$\begin{bmatrix}
    \zeta_8 & 0\\
    0& \zeta_8^{-1}
\end{bmatrix}$, 
we obtain $(D_{n+2})$ from (loc.~cit). Lemma \ref{l:BO-etale} shows $\BO_{\overline{k}}$ agrees with $(E_7)$.

Next, applying Proposition \ref{prop:BD2-1d-irreps}, we see $(E_6)$ from (loc.~cit) 
is contained in $\BT^*_{\overline{k}}$. 
On the other hand, $\BO_{\overline{k}}$ is a constant group scheme of degree $48$ and 
Proposition \ref{prop:BD2-1d-irreps}(ii) 
%\ref{surj-map-to-z2} 
shows that the map 
$\BO(\overline{k})\to \Sym_3\xrightarrow{\textrm{sign}}\ZZ/2$ is surjective, 
so $\BT^*_{\overline{k}}$ is constant with degree $24$. 
Since $(E_6)$ also has degree $24$, we see that it agrees with $\BT^*_{\overline{k}}$. 

Lastly, $\BI$ is defined to be a twist of the classical binary icosahedral group, 
so $\BI_{\overline{k}}$ is isomorphic to the binary icosahedral group. 
So $\BI_{\overline{k}}$ is linearly reductive as it is an \'etale constant group 
with order $120$, which is prime to the characteristic by our assumption 
that $\cha k\neq2,3,5$. 
Hence, by \cite[(3.9)]{Hashimoto}, we see that $\BI_{\overline{k}}$ 
must either agree with $(E_8)$, $\bmu_{120,\overline{k}}$, or 
$\BD_{30,\overline{k}}$ up to conjugation. 
It is not the latter two since $\BI_{\overline{k}}$ is non-abelian and 
does not contain a normal subgroup of order $60$.
\end{proof}

\subsection{Showing all representations descend}
\label{subsec:irreps-decsend}

In order to establish Theorem \ref{thm:gp-sch-over-k-main-features}, it
remains to show assertion (\ref{item-irrep}).

\begin{Lemma}\label{l:summand-over-k-vs-kbar}
Let $\psi$ and $\phi$ be representations of a finite linearly 
reductive group scheme $G$ over a field  $k$ with $\psi$ absolutely irreducible. 
If $\psi\otimes_k \overline{k}$ is a summand of $\phi\otimes_k \overline{k}$, 
then $\psi$ is a summand of $\phi$.
\end{Lemma}

\begin{proof}
We may assume $\psi$ and $\phi$ are not isomorphic over $k$, 
otherwise the statement is clear. 
Thus, we need only prove there is a $G$-equivariant map from $\psi$ to $\phi$. 
Since $\psi\otimes_k \overline{k}$ is irreducible, by assumption, 
and since this is a summand of $\phi\otimes_k \overline{k}$, 
we see $\Hom^G(\psi\otimes_k \overline{k},\phi\otimes_k \overline{k})\neq0$. 
However, 
$$
    \Hom^G\left(\psi\otimes_k \overline{k},\,\phi\otimes_k \overline{k}\right)
    \,=\,
    \Hom^G\left(\psi,\,\phi\right)\otimes_k \overline{k},
$$
hence $\Hom^G(\psi,\phi)\neq0$ by descent.
\end{proof}

\begin{Corollary}\label{cor:descend-quotient-rep}
Let $\psi$ and $\phi$ be representations of a 
finite linearly reductive group scheme $G$ over a field $k$ 
with $\psi$ absolutely irreducible. 
If $\overline{\epsilon}$ is a representation of $G_{\overline{k}}$ with
$$
  \overline{\epsilon}\,\oplus\, (\psi\otimes_k \overline{k})
  \,\cong\, \phi\otimes_k \overline{k},
$$
then $\overline{\epsilon}$ descends to a representation $\epsilon$ 
over $k$ and $\epsilon\oplus\psi\cong\phi$.
\end{Corollary}

\begin{proof}
By Lemma \ref{l:summand-over-k-vs-kbar}, we know $\psi$ is a summand of $\phi$. 
Thus, the quotient representation $\phi/\psi$ base changes over 
$\overline{k}$ to be $\overline{\epsilon}$.
\end{proof}

\begin{proof}[{Proof of Theorem \ref{thm:gp-sch-over-k-main-features}, assertion (\ref{item-irrep})}]
To begin, it is a standard fact that the irreducible representations of $\bmu_n$ are
given by the maps $\bmu_n\to\GG_m$ sending $u\mapsto u^i$. 
Moreover, these representations are defined over any field $k$.

Next, we consider $\BD_n$. 
By \cite[Corollary 2.11]{RDP}, there is a canonical bijection between the irreducible representations $\BD_n$ and the \emph{complex} representations 
of the classical binary dihedral group. 
Thus, we need only write down representations of $\BD_n$ over $k$ which, 
when $k=\CC$, recover the irreducible representations of the 
classical binary dihedral group. 
Note that the complex binary dihedral group $\BD_n(\CC)$ is 
a group with two generators:~$(0,1)$ and $(\zeta_{2n},0)$ 
with $\zeta_{2n}$ a primitive $2n$.th root of unity.
    
We begin with the non-trivial $1$-dimensional representations. 
We have the representation $\BD_n(\CC)\to\bmu_2(\CC)$ given by 
$(\zeta,0)\mapsto 1$ and $(0,1)\mapsto -1$. 
We can extend this to a morphism
$$
  \BD_n^*\,\to\,\bmu_2,\quad (a,b)\,\mapsto\, a^{2n}-b^{2n}.
$$
of group schemes over $k(i)$.
The same proof in Proposition \ref{prop:BD2-1d-irreps} 
shows this is a morphism and that its image lives in $\bmu_2$. 
Then to descend to a representation of $\BD_n$, we must check the map 
is invariant under the action of $\Gal(k(i)/k)$,
which acts as $(a,b)\mapsto (a,-b)$. 
We see this is the case as the exponent on $b$ is even.
    
Next, when $n$ is even, we have two representations 
$\BD_n(\CC)\to\bmu_2(\CC)$ given by $(\zeta,0)\mapsto -1$ and $(0,1)\mapsto \pm1$.
These extend to representations
$$
    \BD_n \,\to\,\bmu_2,\quad (a,b)\,\mapsto\, a^n\pm b^n
$$
over $k$. 
Finally, when $n$ is odd, we have two representations 
$\BD_n(\CC)\to\bmu_4(\CC)$ given by $(\zeta,0)\mapsto -1$ and 
$(0,1)\mapsto \pm i$.
These extend to representations over $k(i)$ by
$$
 \BD_{n,k(i)}^* \,\to\,\bmu_{4,k(i)},\quad (a,b)\,\mapsto\, a^n+(ib)^n.
$$
Note that this map is invariant under the $\Gal(k(i)/k)$-action
since $ib\mapsto (-i)(-b)=ib$. 
Thus, it descends to a representation of $\BD_n$ over $k$.
    
In a similar manner we can extend all of the $2$-dimensional representations of 
$\BD_n(\CC)$ to representations of $\BD_n$. 
These come in two different types. 
The first type is given by $\rho\eta_j$, where $j$ is odd and
$$
    \eta_j \,\colon\,\BD_n\,\to\,\BD_n
$$
descends the map 
$$
    \eta_j^*\,\colon\,\BD_n^*\,\to\,\BD_n^*,\quad (a,b)\,\mapsto\, (a^j,b^j).
$$
Note that $\eta_j^*$ is invariant under the $\Gal(k(i)/k)$-action as $j$ is odd. 
To see $\eta_j^*$ is a morphism, note that
\begin{align*}
    \eta_j^*\left((a,b)(c,d)\right)&=\eta_j^*(ac-bd^{2n-1},ad+bc^{2n-1})\\
    &=\left((ac)^j+(-bd^{2n-1})^j,(ad)^j+(bc^{2n-1})^j\right)\\
    &=\left((ac)^j-(bd^{2n-1})^j,(ad)^j+(bc^{2n-1})^j\right)\\
    &=\eta_j^*(a,b)\cdot\eta_j^*(c,d),
\end{align*}
where the second equality uses that $ab=cd=0$ and the third equality uses that $j$ is odd.
The second type of $2$-dimensional representation is given by descending
$$
    \BD_n^*\,\to\,\SL_2,\quad 
    (a,b)\,\mapsto\,\begin{bmatrix}
    a^{2j} & b^{2j}\\
    b^{2n-2j} & a^{2n-2j}
    \end{bmatrix},
$$
which one verifies is a homomorphism in a similar manner as above. 
Note this is invariant under the $\Gal(k(i)/k)$-action as exponents 
of $b$ are even.

For $\BT$, $\BO$, and $\BI$, the $\overline{k}$-points yield finite groups, 
whose orders are prime to the characteristic (by our characteristic assumptions). 
So, their representations are well-known, see, for example,
\cite[page 339]{Frobenius}.
    
We start by descending the irreducible representations of $\BT(\overline{k})$. 
By Proposition \ref{prop:def-BT}, we have a map $\varphi^+\colon\BT\to\bmu_3$.
Post-composing $\varphi^+$ with the non-trivial automorphism of $\bmu_3$, 
we obtain another map $\varphi^-\colon\BT\to\bmu_3$. 
These yield the two non-trivial one dimensional representations of 
$\BT(\overline{k})$. 
We obtain the two-dimensional representations as $\rho$ and $\rho\otimes\varphi^\pm$.
Thus, using the notation on \cite[page 339]{Frobenius}, we need only construct 
the $3$-dimensional representation corresponding to $\chi^{(3)}$. 
We see $\chi^{(0)}+\chi^{(3)}=(\chi^{(4)})^2$, 
and we have descended the representations corresponding to $\chi^{(0)}$ 
and $\chi^{(4)}$, so Corollary \ref{cor:descend-quotient-rep} 
shows that the representation corresponding to $\chi^{(3)}$ descends as well.

For $\BO$, the unique non-trivial $1$-dimensional representation $\eta$ 
is given by $\BO\to \Sym_3\to \ZZ/2\cong\bmu_2$. 
Two of the $2$-dimensional irreducible representations are $\rho$ and 
$\rho\otimes\eta$. 
The other $2$-dimensional irreducible representation $\epsilon$ is given by 
post-composing $\BO\to \Sym_3$ by the $2$-dimensional irreducible representation 
of $\Sym_3$. 
The $4$-dimensional irreducible representation is given by $\rho\otimes\epsilon$. 
In the notation on \cite[page 339]{Frobenius}, it remains then to construct the
representations corresponding to $\chi^{(2)}$ and $\chi^{(3)}$. 
These are twists of each other by $\eta$, so we need only construct the 
representation corresponding to $\chi^{(3)}$. 
Since $\chi^{(3)}+\chi^{(0)}=(\chi^{(5)})^2$, we are done by 
Corollary \ref{cor:descend-quotient-rep}.

Lastly, we turn to $\BI$. 
We follow the notation in \cite[equation (5)]{ES11} which gives the tensor relations for the representations of $\BI(\overline{k})$. 
The representations $\textbf{1}$ and $\textbf{2}$ 
are the trivial representation and $\rho$, both of which descend. 
The representation $\textbf{2}'$ is the conjugate of $\rho_{\overline{k}}$.
With notation as in Proposition \ref{prop:BI-descends}, it is given 
over $k(\zeta)$ by $\widetilde{\rho}':=\sigma^2\widetilde{\rho}$. 
Since $\sigma\widetilde{\rho}=\widetilde{\rho}\tau$, we see
$$   
\sigma\widetilde{\rho}' \,=\,
\sigma^2(\sigma\widetilde{\rho}) \,=\,
\sigma^2\widetilde{\rho}\tau \,=\,
\widetilde{\rho}'\tau.
$$
As a result, $\widetilde{\rho}'$ descends to $k$.
    
To decsend the remaining representations of $\BI(\overline{k})$, 
we repeatedly apply Corollary \ref{cor:descend-quotient-rep}: 
to descend $\textbf{3}$ and $\textbf{3}'$, we use 
$\textbf{1}\oplus\textbf{3}=\textbf{2}\otimes\textbf{2}$ and 
$\textbf{1}\oplus\textbf{3}'=\textbf{2}'\otimes\textbf{2}'$; 
to descend $\textbf{4}$ and $\textbf{4}'$, we use
$\textbf{4}=\textbf{2}\otimes\textbf{2}'$ and
$\textbf{2}\oplus\textbf{4}'=\textbf{2}\otimes\textbf{3}$; 
to descend $\textbf{5}$ and $\textbf{6}$, we use 
$\textbf{4}\oplus\textbf{5}=\textbf{3}\otimes\textbf{3}'$ and
$\textbf{6}=\textbf{2}\otimes\textbf{3}'$.
\end{proof}

\section{Normalizers of group schemes} 
% over $k$} there was a problem with math mode and hyperref's...

In light of Theorem \ref{thm:NG-and-twists}, in this section we calculate the normalizers of our group schemes over $k$.

\begin{Proposition}\label{prop:NGComputation}
    Let $k$ be a field.
    Then,
    \begin{enumerate}
        \item\label{Nmun-computation} 
        $\Norm_{\SL_2}(\bmu_n)$ is the subgroup scheme of $\SL_{2,k}=\Spec k[a,b,c,d]$ given by $ab=cd=0$ %is given by $\begin{bmatrix}a&b\\c&d\end{bmatrix}$ with $ab=cd=0$ and $ad-bc=1$ 
        for $n\geq2$. 
        \item $\Norm_{\SL_2}(\BD_n^*)=\BD_{2n}$ for $n\geq3$ and $\cha k\neq2$.
        \item $\Norm_{\SL_2}(\BD_2)=\BO$ for $\cha k\neq2$.
        \item $\Norm_{\SL_2}(\BT^*)=\Norm_{\SL_2}(\BO)=\BO$ for $\cha k\neq2,3$.
        \item $\Norm_{\SL_2}(\BI)=\BI$ for $\cha k\neq2,3,5$.
    \end{enumerate}
\end{Proposition}
\begin{proof}

First, a direct check shows that matrices of the form given in 
(\ref{Nmun-computation}) are in $\Norm_{\SL_2}(\bmu_n)$. 
Conversely, let $R$ be a $k$-algebra and let
$$
A\,:=\,\begin{bmatrix}
    a&b\\ c&d
\end{bmatrix}\,\in\, \Norm_{\SL_2}(\bmu_n)(R).
$$ 
Then, consider the matrix 
$$
B\,:=\,\begin{bmatrix}
    u&0\\ 0& u^{n-1}
\end{bmatrix}
\,\in\, \bmu_n\left(R[u]/(u^n-1)\right).
$$ 
Then 
$$
ABA^{-1} \,=\,\begin{bmatrix}
    uad-u^{n-1}bc & ab(u^{n-1}-u)\\
    cd(u-u^{n-1}) & u^{n-1}ad-ubc
\end{bmatrix}
\,\in\,\bmu_n(R[u]/(u^n-1)).
$$
Thus,
$$
ab(u^{n-1}-u) \,=\,cd(u-u^{n-1}) \,=\,0
$$
and since $u$ and $u^{n-1}$ are linearly independent over $R$ in $R[u]/(u^n-1)$, 
we find $ab=cd=0$.

We now prove $\Norm_{\SL_2}(\BD_n^*)=\BD_{2n}$ for $n\geq3$. 
A straightforward computation shows 
$\BD_{2n}(R)\subseteq \Norm_{\SL_2}(\BD_n^*)(R)$. 
Conversely, let 
$$
A\,:=\,\begin{bmatrix}
    a&b\\ c&d
\end{bmatrix} \,\in\, \Norm_{\SL_2}(\BD_n^*)(R).
$$ 
Then consider the matrix 
$$
B \,:=\,\begin{bmatrix}
    \alpha&\beta\\ -\beta^{2n-1}&\alpha^{2n-1}
\end{bmatrix} \,\in\, 
\BD_n^*\left(R[\alpha,\beta]/(\alpha\beta,\alpha^{2n}+\beta^{2n}-1)\right).
$$ 
By assumption, 
$$
\begin{bmatrix}
    \alpha'&\beta'\\\gamma'&\delta'
\end{bmatrix} \,:=\,
ABA^{-1} \,\in\, 
\BD_n^*\left(R[\alpha,\beta]/(\alpha\beta,\alpha^{2n}+\beta^{2n}-1)\right).
$$

Since $\alpha'\beta'=0$, expanding and using $n\geq3$, 
we see the $\alpha^2$ and $\alpha^{2n-2}$ coefficients are respectively 
$a^2bd$ and $ab^2c$. 
Thus, these must vanish. 
Multiplying $ad-bc=1$ by $ab$, we find $ab=0$. 
Similarly, considering the $\alpha^2$ and $\alpha^{2n-2}$ 
coefficients of $\gamma'\delta'$, we find $cd=0$.

We further require $-\gamma'=(\beta')^{2n-1}$, which yields
$$
c^2\beta+d^2\beta^{2n-1} \,=\,
(a^2\beta+b^2\beta^{2n-1})^{2n-1} \,=\,
a^{4n-2}\beta^{2n-1}+b^{4n-2}\beta.
$$
Since $n\geq3$, we have $2n-1\neq 1$, and so,
$$
d^2 \,=\,a^{4n-2},\quad c^2\,=\,b^{4n-2}.
$$
We thus find
$$
a^{4n-1} \,=\, ad^2\,=\,d(ad-bc)\,=\,d\quad\textrm{and} 
-b^{4n-1}\,=\,-bc^2\,=\,c(ad-bc)\,=\,c.
$$
Hence, $A\in\BD_{2n}(R)$.

We now prove $\Norm_{\SL_2}(\BT^*)=\BO$. 
We have $\cha k\neq2$, so $\BD_2$ is \'etale. 
Note that all $R$-valued points of $\BD_2$ have order dividing $4$, 
whereas $\BT^*(R)\setminus\BD_2(R)$ elements have order dividing $6$; 
indeed, this can be verified on $\overline{k}$-points in a manner analogous 
to the proof of Lemma \ref{l:BO-basics}. 
Since the action of $\Norm_{\SL_2}(\BT^*)(R)$ on $\BT^*(R)$ 
preserves the set of elements whose order divides $4$, 
we see it takes $\BD_2(R)$ to itself, that is, 
$\Norm_{\SL_2}(\BT^*)\subset N(\BD_2)=:\BO$. 
Now, since $\BT^*$ is normal in $\BO$, so the action of $\BO$ on 
$\BO$ by conjugation takes $\BT^*$ to itself. 
Hence, $\BO\subset \Norm_{\SL_2}(\BT^*)$ 
and so, $\Norm_{\SL_2}(\BT^*)=\BO$.

To prove $\Norm_{\SL_2}(\BO)=\BO$, we first observe 
$\BO\subseteq \Norm_{\SL_2}(\BO)$. 
To verify this is an equality, we may base change to $\overline{k}$. 
Now $\BO(\overline{k})$ is the classical binary octahedral group.
Aside from the elements of $\BD_2(\overline{k})$, the only other elements 
whose orders divide $4$ are given by a single conjugacy class 
$\mathfrak{c}$ of elements of order precisely $4$. 
We show any automorphism of $\BO(\overline{k})$ (for example, the action of 
$\Norm_{\SL_2}(\BO(\overline{k}))$)
must send $\BD_2(\overline{k})$ to itself and so 
$\Norm_{\SL_2}(\BO)\subset \Norm_{\SL_2}(\BD_2)=\BO$. 
To do so, we need only check that the two conjugacy classes 
$\mathfrak{c}$ and $\mathfrak{c}'$ of order $4$ elements are sent 
to themselves. 
There is one non-trivial $1$-dimensional representation 
(denoted by $\eta$ in the proof of 
Theorem \ref{thm:gp-sch-over-k-main-features}, assertion (\ref{item-irrep})) 
given by 
$\BO(\overline{k})\to\BO(\overline{k})/\BT^*(\overline{k})=\bmu_2(\overline{k})$. 
Thus, any automorphism of $\BO(\overline{k})$ must preserve this non-trivial character.
But one checks that $\mathfrak{c}$ and $\mathfrak{c}'$ 
have different values under this character.

Since $\BI\subset \Norm_{\SL_2}(\BI)$, to prove equality, 
we may base change to $\overline{k}$. 
Now, suppose $A\in \Norm_{\SL_2}(\BI)(\overline{k})$. 
Then, conjugation by $A$ yields an automorphism $\phi$ of $\BI(\overline{k})$. 
We claim it is enough to prove $\phi$ preserves conjugacy classes. 
Indeed, by \cite[Table 9]{BM14}, 
%\matt{published version deleted the table so not sure what to make of that}, 
if $\phi$ preserves conjugacy classes, it is necessarily inner so there exists 
$g\in\BI(\overline{k})$ such that conjugation by $\rho(g)^{-1}A$ 
acts as the identity element on $\rho(\BI(\overline{k}))$. 
With notation as in Proposition \ref{prop:BI-descends}, 
considering the action on $\rho(t)$ shows $\rho(g)^{-1}A$ is diagonal, 
and considering the action on $\rho(r)$ shows $\rho(g)^{-1}A=\pm1$, 
hence $A\in\BI(\overline{k})$. 
It therefore remains to prove $\phi$ preserves conjugacy classes. 
For this, note that all conjugacy classes are determined by the orders 
of the elements in the conjugacy class aside from those of order $5$ 
(which has two distinct conjugacy classes given by $t^2$ and $t^4$) 
and those of order $10$ (given by $t$ and $t^3$).
Thus, if $\phi$ does not preserve conjugacy classes then there exists 
$g\in\BI(\overline{k})$ such that $A\rho(t^i)A^{-1}=\rho(gt^{i\pm2}g^{-1})$ 
for some $i$. 
Rearranging, we have $\rho(g)^{-1}A\rho(t^i)=\rho(t^{i\pm2})\rho(g)^{-1}A$, 
and an easy computation shows there is no matrix $B$ in $\SL_2(\overline{k})$ 
for which $B\rho(t^i)=\rho(t^{i\pm2})B$.
\end{proof}

\section{Normalizers and the McKay graph:~Theorem \ref{thm:NGmodGAutMcKay}}

Let $\rho\colon G\to\SL_{2,k}$ be the inclusion of a finite linearly reductive subgroup 
scheme over a field $k$.  
Via the map $\Norm_{\SL_2}(G)(\overline{k})\to\AutScheme_G(\overline{k})$, 
we see that $\Norm_{\SL_2}(G)(\overline{k})$ 
acts on the (isomorphism classes of) 
irreducible representations of $G_{\overline{k}}$. 
Furthermore, by construction, this action fixes $\rho_{\overline{k}}$ 
(up to equivalence), so we obtain an action on the McKay graph of 
$\rho_{\overline{k}}$, that is, we have a homomorphism of groups
$$
\Norm_{\SL_2}(G)(\overline{k}) \,\to\,
\Aut\left(\Gamma(G_{\overline{k}},\rho_{\overline{k}})\right).
$$
Note that $G(\overline{k})$ is in the kernel since if $h\in G(\overline{k})$ 
and $\rho_i\colon G_{\overline{k}}\to\GL_{n,\overline{k}}$ 
is any representation, then the automorphism 
$\varphi_h$ given by conjugation by $h$ satisfies, 
for all $\overline{k}$-algebras $R$ and all $g\in G(R)$,
$\rho_i(\varphi_h(g))=\rho_i(hgh^{-1})=\rho_i(h)\rho_i(g)\rho_i(h)^{-1}$, 
so $\rho_i$ is equivalent to $\rho_i\varphi_h$. 
Thus, we find a homomorphism of groups
$$
\left(\Norm_{\SL_2}(G)/G\right)(\overline{k})
\,=\,
\Norm_{\SL_2}(G)(\overline{k})/G(\overline{k})
\,\to\,
\Aut\left(\Gamma(G_{\overline{k}},\rho_{\overline{k}})\right).
$$

The main goal of this subsection is to prove Theorem \ref{thm:NGmodGAutMcKay}, 
which shows that this map is an isomorphism in our cases of interest. 
In fact, we prove a slightly more general version of Theorem \ref{thm:NGmodGAutMcKay}.

\begin{Theorem}\label{thm:NGmodGAutMcKay-more-general}
    Let $k$ be a field and let $G\subset\SL_{2,k}$ with 
    \begin{enumerate}
        \item $(G,\cha k)$ as in \eqref{eqn:ADE-over-perfect-k} of
        Theorem \ref{thm:gp-sch-over-k-main-features}, 
        or
        \item $G=\BD_n^*$ with $\cha k\neq2$, or
        \item $G=\BT^*$ with $\cha k\neq2,3$.
    \end{enumerate}
    Then, $\Norm_{\SL_{2,k}}(G)/G$ is a finite constant group scheme and 
    we have a natural isomorphism 
    \begin{equation}\label{eqn:NmodG-McKay}
        \Norm_{\SL_{2,k}}(G)/G
        \quad\xrightarrow{\simeq}\quad
        \Aut\left(\Gamma(G_{\overline{k}},\rho_{\overline{k}})\right).
    \end{equation}
\end{Theorem}

\begin{proof}
We begin by showing $\Norm_{\SL_2}(G)/G$ is a constant group scheme. 
When $G=\BD_n^*$ with $n\geq3$, we have 
$$
\Norm_{\SL_2}(\BD_n^*)/\BD_n^* \,=\,
\BD_{2n}^*/\BD_n^* \,\xleftarrow{\cong}\,
\bmu_{4n}/\bmu_{2n} \,\cong\, \bmu_2\,\cong\,\ZZ/2,
$$
where the rightmost isomorphism uses that $\cha k\neq2$. 
For $G=\BD_2$, we have a morphism $\Norm_{\SL_2}(\BD_2)=\BO\to\Sym_3$. 
Proposition \ref{prop:BD2-1d-irreps} shows the kernel is $\BD_2$, 
so we know $\BO/\BD_2\cong\Sym_3$. 
For $G\in\{\BO,\BI\}$, we know $\Norm_{\SL_2}(G)/G$ is trivial, hence constant. 
Lastly, for $G=\BT^*$, we know $\Norm_{\SL_2}(\BT^*)/\BT^*=\BO/\BT^*$ 
is \'etale since $\BO$ and $\BT^*$ are.
Moreover, over $\overline{k}$, the quotient is the group $\ZZ/2$. 
Thus, $\Norm_{\SL_2}(\BT^*)/\BT^*$ is a twisted form of $\ZZ/2$.
But $\ZZ/2$ has no non-trivial twisted forms since $\Aut(\ZZ/2)$ is trivial.

To handle the case of $G\in\{\BD_{2n+1},\BT\}$, by construction, 
there is a Galois extension $L/k$ with $G_L=A G^*_L A^{-1}$ 
for some $A\in\SL_2(L)$. 
Then, we see 
$(\Norm_{\SL_{2,k}}(G))_L=\Norm_{\SL_{2,L}}(G_L)=A \Norm_{\SL_{2,L}}(G^*_L) A^{-1}$. 
Thus, the inclusion $G\subseteq\Norm_{\SL_2}(G)$ is a Galois twist of
$G^*\subseteq\Norm_{\SL_2}(G^*)$. 
It follows that $\Norm_{\SL_2}(G)/G$ is a twist of 
$\Norm_{\SL_2}(G^*)/G^*\cong\ZZ/2$. But $\ZZ/2$ has no non-trivial twists, so $\Norm_{\SL_2}(G)/G\cong\ZZ/2$ as well.

Having now established $\Norm_{\SL_2}(G)/G$ is a constant group scheme, 
we may pass to $\overline{k}$ and construct our natural isomorphism with
$\Aut(\Gamma(G_{\overline{k}},\rho_{\overline{k}}))$. 
Note that the cases $G\in\{\BD_{2n+1},\BT\}$ will follow from the case of 
$G^*$ since $G_{\overline{k}}$ is isomorphic to $G^*_{\overline{k}}$.

First consider the case $G=\BD_n^*$ with $n\geq3$ and $\cha k\neq2$. 
By Proposition \ref{prop:NGComputation}, $\Norm_{\SL_2}(\BD_n^*)=\BD_{2n}^*$. 
The natural maps
$$
\BD_{2n}^*/\BD_n^* \,\xleftarrow{\cong}\,
\bmu_{4n}/\bmu_{2n} \,\cong\,\bmu_2 \,\cong\, \ZZ/2
$$
are isomorphisms, where the rightmost isomorphism uses that $\cha k\neq2$. 
Thus, we see $\BD_{2n}^*/\BD_n^*$ has precisely one non-trivial 
$\overline{k}$-point given by the class of $(\alpha,0)$ 
for any $\alpha\in\bmu_{4n}(\overline{k})\setminus\bmu_{2n}(\overline{k})$.
Explicitly conjugating the matrix $\rho((a,b))$ by the matrix $\rho((\alpha,0))$ 
shows $(\BD_{2n}^*/\BD_n^*)(\overline{k})$ acts on $\BD_n^*$ 
by $(a,b)\mapsto(a,\alpha^2 b)$. 
Since $\cha k\neq2$, we know $\bmu_4(\overline{k})\simeq\ZZ/4$.
Let $i$ be a primitive generator. 
As in the proof of Theorem \ref{thm:gp-sch-over-k-main-features}, assertion 
(\ref{item-irrep}), we have representations
$$
\varphi^\pm\quad\colon\quad
\BD_{n,\overline{k}}^* \,\to\,
\bmu_{4,\overline{k}},\quad (a,b)\,\mapsto\, a^n\pm (ib)^n.
$$
If $n$ is even, then the image lands in $\bmu_{2,\overline{k}}$. 
We see
$$
\varphi^\pm(a,\alpha^2 b)
\,=\,
a^n \pm\alpha^{2n}(ib)^n\,=\,a^n\mp(ib)^n \,=\,\varphi^\mp(a,b).
$$
Thus, the unique non-trivial element of $(\BD_{2n}^*/\BD_n^*)(\overline{k})$ 
maps to a non-trivial element of 
$\Aut(\Gamma(\BD^*_{n,\overline{k}},\rho_{\overline{k}}))$. 
Furthermore, \cite[Corollary 2.11]{LRQ} shows there is a canonical bijection 
between the irreducible representations $\BD_n^*$ and the complex representations 
of the classical binary dihedral group, respecting tensor products. 
Hence, $\Gamma(\BD^*_{n,\overline{k}},\rho_{\overline{k}})$ is 
isomorphic to the affine Dynkin diagram $\widetilde{D}_{n+2}$ 
with the vertex corresponding to the trivial representation removed, 
that is, $D_{n+2}$, so its automorphism group is $\ZZ/2$. 
As a result, \eqref{eqn:NmodG-McKay} is an isomorphism when $G=\BD_n^*$ for 
$n\geq3$.

We now consider $\BD_2$. 
By the proof of Theorem \ref{thm:gp-sch-over-k-main-features},
assertion (\ref{item-irrep}), 
$\BD_{2,\overline{k}}$ has $4$ non-trivial irreducible representations 
obtained as the base changes to $\overline{k}$ of $\rho$,
as well as the characters $\psi^\pm$ and $\psi^0$ from 
Proposition \ref{prop:BD2-1d-irreps}. 
By \cite[Corollary 2.11]{LRQ}, the McKay graph with the trivial representation 
removed has three edges:~those connect $\rho_{\overline{k}}$ to each of 
$\psi^\pm_{\overline{k}}$ and $\psi^0_{\overline{k}}$.
Thus, $\Aut(\Gamma(G_{\overline{k}},\rho_{\overline{k}}))\cong\Sym_3$ 
given by the permutation action on the three characters. 
On the other hand, Proposition \ref{prop:BD2-1d-irreps} shows that the 
permutation action of 
$\BO_{\overline{k}}:=\Norm_{\SL_2}(\BD_{2,\overline{k}})$ 
on the characters is transitive with kernel $\BD_2(\overline{k})$.
Transitivity follows from the fact that 
$\BD_4(\overline{k}) \setminus \BD_2(\overline{k})$ and 
$\frac{1}{\sqrt{2}}\begin{bmatrix}
    1 & 1\\
    -1 & 1
\end{bmatrix}$ 
act as distinct transpositions and so, they generate $\Sym_3$. 
Thus, \eqref{eqn:NmodG-McKay} is an isomorphism.

Next, by the definition of $\BT^*$ in \S\ref{subsec:def-our-gp-schs}, 
we have a map $\varphi^+\colon\BT^*\to\ZZ/3$ where $\ZZ/3$ is the kernel of 
the map $\Sym_3\xrightarrow{\textrm{sign}} \ZZ/2$. 
Post-composing $\varphi^+$ with the non-trivial automorphism of $\ZZ/3$, 
we obtain another map $\varphi^-\colon\BT^*\to\ZZ/3$. 
Since we are assuming $\cha k\neq3$, we may identify $\ZZ/3$ and 
$\bmu_3(\overline{k})$. 
These yield the two non-trivial characters of $\BT^*(\overline{k})$, 
see Proposition \ref{prop:def-BT} and the proof of 
Theorem \ref{thm:gp-sch-over-k-main-features}, assertion (\ref{item-irrep}). 
Now, \cite[Corollary 2.11]{LRQ} and Proposition \ref{prop:NGComputation} 
show that both the source and target of \eqref{eqn:NmodG-McKay} are isomorphic 
to $\ZZ/2$.
Furthermore, Proposition \ref{prop:BD2-1d-irreps} shows any generator of 
$\bmu_8(\overline{k})$ (which lives in 
$\BD_4(\overline{k}) \setminus \BD_2(\overline{k})$) 
swaps $\varphi^\pm$. 
Thus, \eqref{eqn:NmodG-McKay} is an isomorphism.

Finally, for $\BO$ and $\BI$, \cite[Corollary 2.11]{LRQ} and 
Proposition \ref{prop:NGComputation} show that both the source and 
target of \eqref{eqn:NmodG-McKay} are trivial.
\end{proof}

The following corollary will be useful in classifying twisted forms of 
the RDP singularities over perfect fields.

\begin{Corollary}\label{cor:twisted-invariants-NGmodG}
Let $\rho\colon G\to\SL_{2,k}$ be the inclusion of a 
non-abelian finite linearly reductive group scheme,
where $k$ is a perfect field.
Then, there exists a finite Galois extension $L/k$, an injection
$$
\iota \,\colon\,\Gal(L/k) \,\hookrightarrow\,
\Aut\left(\Gamma(G_{\overline{k}},\rho_{\overline{k}})\right)
$$
and an isomorphism of $k$-algebras
$$
k[x,y]^G \,\cong\, \left(L[x,y]^H\right)^{\Gal(L/k)},
$$
where the $\Gal(L/k)$-action is the diagonal action on $L[x,y]^H$ given by the Galois-action
on $L$ and conjugation via $\imath$,
where $H\subset\SL_{2,L}$ is one of
$$
\BD_n^*,\quad \BT^*,\quad \BO,\quad \BI,
$$
and $G_{\overline{k}}$ is $\SL_2(\overline{k})$-conjugate to $H_{\overline{k}}$.
\end{Corollary}

\begin{proof}
For ease of notation, we set $\Gal_M:=\Gal(\overline{M}/M)$ 
for any field $M$. 
By Theorem \ref{thm:gp-sch-over-k-main-features}, assertion (\ref{item-conjugation}), 
we know that $G\subset\SL_{2,k}$ is a twisted form of one of the 
subgroup schemes appearing in the theorem. 
Thus, by Corollary \ref{cor: brian} and Remark \ref{rem: brian}, 
it is classified by a $1$-cocycle 
$c\colon\Gal_k\to \Norm_{\SL_{2,k}}(H)(\overline{k})$. 
By Theorem \ref{thm:NGmodGAutMcKay}, 
$\Norm_{\SL_{2,k}}(H)/H$ is a constant group scheme over $k$, 
so
$$
H^1(\Gal_k,\,\Norm_{\SL_{2,\overline{k}}}(H(\overline{k}))/H(\overline{k})) \,\cong\, 
\Hom(\Gal_k,\,\Norm_{\SL_{2,\overline{k}}}(H(\overline{k}))/H(\overline{k})).
$$
As a result, the composite map
$$
\Gal_k \,\xrightarrow{c}\, \Norm_{\SL_{2,k}}(H)(\overline{k}) \,\to\,
\left(\Norm_{\SL_{2,k}}(H)/H\right)(\overline{k})
\,\cong\,\Aut\left(\Gamma(H_{\overline{k}},\rho_{\overline{k}})\right)
$$
is a homomorphism of groups.
Hence, the kernel is of the form $\Gal(\overline{k}/L)$ 
for some finite extension Galois $L/k$. 
Then, the induced map
$$
\Gal(L/k) \,\hookrightarrow\,
\Aut\left(\Gamma(G_{\overline{k}},\rho_{\overline{k}})\right)
$$
is an injection. 
Next, the cocycle $c$ yields an action of 
$\Gal_k$ on $H_{\overline{k}}$ and we compute
\begin{align*}
k[x,y]^G  &\cong\,
\overline{k}[x,y]^{H_{\overline{k}}\rtimes\Gal_k}
\,=\,\left((\overline{k}[x,y]^{H_{\overline{k}}})^{\Gal_L}\right)^{\Gal(L/k)}\\
&=\,\left(\left(\overline{k}\otimes_L L[x,y]^H\right)^{\Gal_L}\right)^{\Gal(L/k)}
\,=\, (L[x,y]^H)^{\Gal(L/k)},
\end{align*}
where the last equality comes from the fact that $\Gal_L$ 
factors through $H(\overline{k})$, hence acts trivially on the invariant ring. 
\end{proof}

\section{Equations of twisted RDP singularities:~Theorem \ref{tm:twisted-invariants}}

We now classify all singularities types $\Aff^2_k/G$, where where 
$G\subset\SL_{2,k}$ is a finite linearly reductive group scheme, and
where $k$ is a perfect field.
We do so by applying Corollary \ref{cor:twisted-invariants-NGmodG},
which reduces us to considering all possible injective maps from 
$\Gal(L/k)$ to the automorphism of the McKay graph.
Regarding the definitions of the RDP singularities over nonclosed
fields, we refer to \cite[\S24]{Lipman}.

\subsection{Types $A$ and $B$}

Consider $H=\bmu_n\subset\SL_{2,k}$ with $n\geq3$.
%\matt{I think we can generalize the corollary to handle this too: we have $\Gal_k\to N\bmu_n/\GG_m\to \ZZ/2$ again constant group scheme, so we again get the kernel and then the final sentence of the proof is that action factors through $\GG_m$ which acts trivially on the invariant ring b/c diagonal matrices commute} 
Although Corollary \ref{cor:twisted-invariants-NGmodG} does not directly apply 
(for example, $\Norm_{\SL_2}(\bmu_n)/\bmu_n$ is not constant), 
we can follow the same argument in the proof. 
For a twisted subgroup scheme $G\subset\SL_{2,k}$, we have
$$
k[x,y]^G \,=\,\overline{k}[x,y]^{\bmu_n\rtimes\Gal_k},
$$
where $\Gal_k$ acts on $\bmu_n$ via
$$
\Gal_k \,\to\, \Norm_{\SL_2}(\bmu_n)(\overline{k}) \,\cong\,
\GG_m(\overline{k})\rtimes\ZZ/2
$$
by Proposition \ref{prop:NGComputation}. 
We see that $\GG_m(\overline{k})$ acts trivially on $\bmu_n(\overline{k})$ 
(since diagonal matrices commute) and $\ZZ/2$, 
which is generated by the matrix $\begin{bmatrix}
    0&1\\-1&0
\end{bmatrix}$, 
acts on $\bmu_n$ by $u\mapsto u^{-1}$ and swaps $x,y$. 
\subsubsection{$\cha k\neq2$}
Let us now assume $\cha k\neq2$.
Then, we are reduced to looking at quadratic extensions $L=k(\sqrt{d})$ and we see
$$
k[x,y]^G \,=\,L[x,y]^{\bmu_n\rtimes\Gal(L/k)}=L[x^n,xy,y^n]^{\ZZ/2}
$$
where $\ZZ/2$ swaps $x,y$, and sends $\sqrt{d}$ to $-\sqrt{d}$. 
So,
$$
k[x,y]^G \,=\, k[x^n+y^n,xy,\sqrt{d}(x^n-y^n)].
$$
Letting
$$
A\,:=\,x^n+y^n,\quad B\,:=\,\sqrt{d}(x^n-y^n),\quad C\,:=\,xy,
$$
we see
$$
dA^2-B^2=d(x^{2n}+2x^ny^n+y^{2n})-d(x^{2n}-2x^ny^n+y^{2n})=4d(xy)^n=4dC^n.
$$
If $\sqrt{d}\in k$, then after blowing up $n$ times, one obtains that  
is an RDP of type $A_{n-1}$.
If $\sqrt{d}\notin k$, then after blowing up $\beta(n)$ times, one obtains that 
this is an RDP of type 
$B_{\beta(n)}$, with $\beta(-)$ as defined in 
Theorem \ref{tm:twisted-invariants},
see also the example \cite[\S24, page 263]{Lipman}.
We also note that this singularity of type $B_{\beta{n}}$
splits over $k(\sqrt{d})$, that is,
it becomes isomorphic to an RDP of type $A_{n-1}$.
Moreover, it is see that two such equations with parameters $d,d'\in k^\times$
are isomorphic over $k$ if and only if $d/d'\in k^{\times2}$.

\subsubsection{$\cha k=2$}
Finally, assume $\cha k=2$ and then, we have to consider a 
quadratic extension of the form $L=k(\alpha)$,
where $\alpha$ is a root of the Artin--Schreier equation
$\wp(x):=x^2-x=d$. 
We then obtain
$$
k[x,y]^G \,=\, L[x,y]^{\bmu_n\rtimes\Gal(L/k)}\,=\,L[x^n,xy,y^n]^{\ZZ/2},
$$
where $\ZZ/2$ swaps $x,y$ and sends $\alpha$ to $\alpha+1$. 
So,
$$
k[x,y]^G \,=\, k[x^n+y^n,xy,\alpha x^n+(\alpha+1)y^n)].
$$
Letting
$$
A\,:=\,x^n+y^n,\quad B\,:=\,\alpha x^n + (\alpha+1)y^n,\quad C\,:=\,xy,
$$
we see
$$
  dA^2+AB+B^2 \,=\, C^n.
$$
If $d\in\wp(k)$, then this is an RDP of type $A_{n-1}$ and
if $d\notin\wp(k)$, then this is an RDP of type 
$B_{\beta(n)}$, see again \cite[\S24, page 263]{Lipman}.
In the latter case, it splits over $k(\wp^{-1}(d))$, where
it becomes an RDP ot type $A_{n-1}$.
Moreover, it is see that two such equations with parameters $d,d'\in k$
are isomorphic over $k$ if and only if $d-d'\in\wp(k)$.

\subsection{Types $C_{n+1}$ and $D_{n+2}$ for $n\geq3$}\label{subsec:CDtwist}

Here, we consider $\BD_n^*$ for $n\geq3$
and in particular, we have $\cha k\neq2$.
Corollary \ref{cor:twisted-invariants-NGmodG} reduces us to considering 
$L=k(\sqrt{d})$ and the unique non-trivial map
$$
\Gal(L/k) \,\to\, \left(\Norm_{\SL_2}(\BD^*_n)/\BD^*_n\right)(\overline{k}).
$$
It takes the generator $\sigma\in\Gal(L/k)$ to the unique 
non-trivial element $(\Norm_{\SL_2}(\BD^*_n)/\BD^*_n)(\overline{k})$, 
which we know by Proposition \ref{prop:NGComputation}, 
is given by (the coset defined by) the diagonal 
matrix $\diag(\zeta_{4n},\zeta_{4n}^{-1})$, where $\zeta_{4n}\in\overline{k}$
is a primitive $4n$.th root of unity.
(If $p\mid n$, then one has to work with the nonreduced 
group scheme $\bmu_{4n}\subset\SL_{2,k}$ 
that acts on $\BD_n^*\subset\SL_{2,k}$ by conjugation.
Up to inner automorphism of $\BD_n^*$, this yields
an element of order $2$ 
% in the outer automorphism group scheme $\OutScheme_{\BD_n^*}$ 
and thus, an element of order $2$
in $(\Norm_{\SL_2}(\BD^*_n)/\BD_n^*)(\overline{k})$,
which justifies the above reasoning also in this case.)

So, we have
$$
L[x,y]^{\BD^*_n\rtimes\Gal(L/k)} \,=\,L[u,v,w]^{\Gal(L/k)}
$$
where 
\begin{equation}\label{BD2-invariants}
u:=x^{2n}+y^{2n},\quad v:=(xy)^2,\quad w:=xy(x^{2n}-y^{2n}).
\end{equation}
They satisfy the equation
\begin{equation}\label{Dn-sing}
w^2 \,=\,(u^2-4v^n)v.
\end{equation}
Since $\diag(\zeta_{4n},\zeta_{4n}^{-1})$ sends $x,y$ to 
$\zeta_{4n} x,\zeta_{4n}^{-1}y$, 
we see 
$$
\sigma(\sqrt{d})\,=\,-\sqrt{d},\quad \sigma(u)\,=\,-u,\quad 
\sigma(v)\,=\,v,\quad \sigma(w)\,=\,-w.
$$
Then
$$
L[x,y]^{\BD^*_n\rtimes\Gal(L/k)} \,=\, k[U,V,W]
$$
where
\begin{equation}\label{BDn-twist-invariants}
U \,:=\,\sqrt{d} u,\quad V\,:=\,v,\quad W\,:=\,\sqrt{d} w.
\end{equation}
These invariants satisfy the equation
\begin{equation}\label{Dn-sing-twist}
W^2=(U^2-4dV^n)V.
\end{equation}
If $\sqrt{d}\in k$, then this is an RDP of type $D_{n+2}$
and if $\sqrt{d}\notin k$, then this is an RDP of type
$C_{n+2}$.
In the latter case, it splits over $k(\sqrt{d})$,
that is, becomes isomorphic to an RDP of type $D_{n+2}$.
Moreover, two such equations with parameters $d,d'\in k^\times$
are isomorphic over $k$ if and only if $d/d'\in k^{\times 2}$.

\subsection{Types $F_4$ and $E_6$}
Before computing the twisted singularities, we have to understand
quotients by $\BT$, which lead to $E_6$-singularities.
It turns out that quotients by $\BT^*$ are easier to handle
and to obtain $\BT$-quotients after quadratic twist.
Let $k$ be a field with $\cha k\neq2,3$.
By Proposition \ref{prop:BD2-1d-irreps}, we have
an exact sequence
$$
1 \,\to\, \BD_2 \,\to\, \BT^* \,\to\, \ZZ/3 \,\to\, 1,
$$
so 
\begin{equation}
\label{eq: BTast invariants}
k[x,y]^{\BT^*} \,=\, k[u,v,w]^{\ZZ/3}
\end{equation}
where $u=x^4+y^4$, $v=(xy)^2$ and $w=xy(x^4-y^4)$. 
Then $\ZZ/3$ acts as the matrix
$$
-\frac{1}{\sqrt{2}}
\begin{bmatrix}
    \zeta_8&\zeta_8\\
    -\zeta_8^{-1}&\zeta_8^{-1}
\end{bmatrix}
$$
where $\zeta_8\in\overline{k}$ is a primitive $8$.th root of unity.
$$
\tau(u) \,=\,-\frac{1}{2}(u-6v),\quad \tau(v)\,=\,-\frac{1}{4}(u+2v),\quad\tau(w)\,=\,w.
$$
Note that the eigenvalues of $\tau$ acting on the $2$-dimensional 
vector space $Lu\oplus Lv$ are given by the $3$rd roots of unity that we 
denote $\zeta_3^{\pm}$.
(When working with $\BT$, we would have to understand the quotient by 
some $\bmu_3$-action on $k[u,v,w]$ in \eqref{eq: BTast invariants},
which is harder than working with the above $\ZZ/3$-action.)

Now, over $k(\sqrt{-3})$, we may diagonalize. 
Explicitly, the eigenvectors are
\begin{equation}\label{eqn:xipm}
\xi^\pm \,:=\,u-(2+4\zeta_3^\mp)v \,=\, u\mp 2\sqrt{-3}v,
\end{equation}
so $k(\sqrt{-3})[u,v,w]^{\ZZ/3}$ is given by the invariants
\begin{equation}\label{eqn:E6-invariants-explicit}
A^\pm \,:=\,(\xi^\pm)^3,\quad B \,:=\,\xi^+\xi^-=u^2+12v^2,\quad C\,:=\,w
\end{equation}
and, in fact, $A^-$ is not needed since
\begin{equation}\label{eqn:A+A-C2}
\sqrt{-3}(A^+-A^-)\,=\,(6C)^2.
\end{equation}
These invariants satisfy the unique relation
\begin{equation}\label{eqn:BT-invariants1}
(A^+)^2 + 12A^+C^2\sqrt{-3} \,=\, B^3.
\end{equation}
Note this is the same as
\begin{equation}\label{eqn:BTast-sing}
(A^++6C^2\sqrt{-3})^2 + 108C^4\,=\,B^3.
\end{equation}
If $\sqrt{-3}\in k$, then \eqref{eqn:BTast-sing} 
defines the $\BT^*$-invariants. 
If $\sqrt{-3}\notin k$, then
$$
k[x,y]^{\BT^*} \,=\,k(\sqrt{-3})[A^+,B,C]^{\Gal(k(\sqrt{-3})/k)}.
$$
Note $\Gal(k(\sqrt{-3})/k)$ fixes $B$ and $C$, and maps $A^+$ to $A^-$. 
Thus, 
$$
k[x,y]^{\BT^*} \,=\,k[A,B,C]
$$
where
\begin{equation}\label{eqn:BT-eqA}
A \,:=\,A^+ + A^- \,=\, 2u(u^2-36v^2).
\end{equation}
Note in light of \eqref{eqn:A+A-C2} that the Galois invariant expression 
$\sqrt{-3}(A^+-A^-)$ is not needed. 
Using  \eqref{eqn:A+A-C2} and \eqref{eqn:BT-eqA}, we see
$$
A^+ \,=\,\frac{A}{2}-6\sqrt{-3}C^2,
$$
so \eqref{eqn:BTast-sing} becomes
\begin{equation}\label{eqn:BT-invariants2}
(A/2)^2+108C^4 \,=\, B^3.
\end{equation}
Note that this is the same equation as \eqref{eqn:BTast-sing} after a 
change of variables that replaces $A$ by $2A$ and $A^+$ by $A^+-6C^2\sqrt{-3}$.
Summing up, we find that
\begin{equation}
\label{eq:BTast equation}
k[x,y]^{\BT^*}\,\cong\,k[R,S,T]/(R^2+108S^4-T^3)
\end{equation}
for every field $k$ with $\cha k\neq2,3$.

To compute twisted forms over perfect fields $k$, 
Corollary \ref{cor:twisted-invariants-NGmodG} implies that we
need only consider $L=k(\sqrt{d})$ with $\sqrt{d}\not\in k$ and
a non-trivial map
$$
\Gal(L/k) \,\to\, \left(\Norm_{\SL_2}(\BT^*)/\BT^*\right)(\overline{k}),
$$
which necessarily takes the generator $\sigma\in\Gal(L/k)$ to the matrix 
$\diag(\zeta_8,\zeta_8^{-1})$ where $\zeta_8\in\overline{k}$ is a primitive
$8$.th root of unity.
This uses our explicit description of $\Norm_{\SL_2}(\BT^*)$ 
given in Proposition \ref{prop:NGComputation}. 
So,
$$
\sigma(x) \,=\,\zeta_8 x,\quad \sigma(y)\,=\,\zeta_8^{-1} y
$$
hence
$$
\sigma(u)=-u,\quad \sigma(v)=v, \quad \sigma(w)=-w,\quad \sigma(A)=-A,\quad \sigma(B)=B,\quad \sigma(C)=-C.
$$

First, consider the case where $\sqrt{-3}\notin L$. 
Then 
$$
L[x,y]^{\BT^*\rtimes\Gal(L/K)} \,=\,
L[A,B,C]^{\Gal(L/K)}=k[\widetilde{A},B,\widetilde{C}]
$$
where
\begin{equation}
    \widetilde{A}\,:=\,\sqrt{d}A,\quad \widetilde{C}\,:=\,\sqrt{d}C.
\end{equation}
From \eqref{eqn:BT-invariants2}, the invariants satisfy the equation
$$
\frac{1}{4d}\widetilde{A}^2+\frac{108}{d^2}\widetilde{C}^4 \,=\, B^3,
$$
or equivalently
\begin{equation}\label{eqn:BT-twisted2}
(d\widetilde{A}/2)^2 \,+\, 108d\widetilde{C}^4 \,=\, (dB)^3.
\end{equation}

Now assume $\sqrt{-3}\in L$. If $\sqrt{-3}\notin k$, then we may assume $d=-3$. We see 
$$
\sigma(A^\pm) \,=\,-u\mp2(-\sqrt{-3})v\,=\,-A^\pm.
$$
So letting
\begin{equation}
    \widetilde{A}^+:=\sqrt{d}A^+=\sqrt{-3}A^+
\end{equation}
we see 
$$
L[x,y]^{\BT^*\rtimes\Gal(L/k)} \,=\,L[A^+,B,C]^{\Gal(L/k)} 
\,=\,k[\widetilde{A}^+,B,\widetilde{C}];
$$
using equation \eqref{eqn:BTast-sing}, we see
%which from \eqref{eqn:BT-invariants1} satisfy the equation
$$
-\frac{1}{3}(\widetilde{A}^++6\widetilde{C}^2)^2+12\widetilde{C}^4 \,=\, B^3.
$$
Multiplying through by $-27$, we see
\begin{equation}\label{eqn:BT-twisted1a}
9(\widetilde{A}^++6\widetilde{C}^2)^2-4(3\widetilde{C})^4\,=\,(-3B)^3.
\end{equation}
Up to a change of variables, this is the same equation as
\begin{equation}
\label{eq:BTequation1}
T^2-4V^4\,=\,U^3.
\end{equation}
Note that this equation, up to changing $V$ by a factor of $3$, has the same form as 
$$
T^2+108(-3)V^4\,=\,U^3,
$$
which is the equation for the singularity obtained from \eqref{eqn:BT-twisted2} 
if one were to substitute $d=-3$.

Lastly, assume $\sqrt{-3}\in k$, so $\sigma(A^\pm)=-A^\mp$. 
Note then that $A^+\mp A^-$ have eigenvalues $\pm1$ with respect to $\sigma$. 
Thus, the invariants are $A^+-A^-$, as well as 
$$
B,\quad\sqrt{d}C=:\widetilde{C},\quad \sqrt{d}(A^++A^-)=:\widetilde{A}.
$$
Now $A^+-A^-$ is expressible in terms of $\widetilde{C}$ by \eqref{eqn:A+A-C2} 
so it is not needed. We therefore obtain the same equation \eqref{eqn:BT-twisted2} 
as before.
Summing up, we find that the quadratic twist by $L=k(\sqrt{d})$ is
given by
\begin{equation}
\label{eq:BTast twist equation}
L[x,y]^{\BT^*\rtimes\Gal(L/k)}\,\cong\,k[R,S,T]/(R^2+108dS^4-T^3).
\end{equation}

Finally, the group scheme $\BT$ becomes isomorphic to $\BT^*$
over $L=k(\sqrt{-3})$.
In particular, $\BT$-invariants
correspond to $\BT^*$-invariants if $\sqrt{-3}\in k$
and since $\sqrt{-3}\in k$, the equation \eqref{eq:BTast equation}
simplifies to \eqref{eq:BTequation1}.
If $\sqrt{-3}\notin k$, then we have to use \eqref{eq:BTast twist equation}
with $d=-3$.
Summing up, we find
\begin{equation}
    \label{eq:BT equation}
    k[x,y]^{\BT} \,\cong\, k[R,S,T]/(R^2-4S^4-T^3)
\end{equation}
for every perfect field of characteristic $\neq2,3$.
This is an RDP of type $E_6$.

Using \eqref{eq:BTast twist equation}, we can also compute
quadratic twists of this: if $L=k(\sqrt{d})$, then
$$
L[x,y]^{\BT\rtimes\Gal(L/k)} \,\cong\, k[R,S,T]/(R^2-4dS^4-T^3).
$$
If $\sqrt{d}\in k$, then this is an RDP of type $E_6$
and if $\sqrt{d}\notin k$,
then this is an RDP of type $F_4$.
In the latter case, it splits over $L$, that is, becomes isomorphic to an RDP
of type $E_6$.
Moreover, two such equations with parameters $d,d'\in k^\times$
are isomorphic over $k$ if and only if $d/d'\in k^{\times 2}$.

\subsection{Types $D_4$ and $G_2$}\label{subsec:S3twist}
Let $k$ be a field with $\cha\neq2$.
By Corollary \ref{cor:twisted-invariants-NGmodG}, twists of the $D_4$-singularity 
are given by Galois extensions, whose Galois group is a subgroup of 
$\Sym_3$. 
Here, we consider $L/k$ a Galois extension with group $\Sym_3$ or $\ZZ/3$
and we will deal with $\ZZ/2$ in \S\ref{subsec:D4quadtwist}.
Then, $L$ is the splitting field of a cubic, which we can assume to be
of the form $t^3+at+b\in k[t]$:
If $\cha k\neq3$, then this is via the usual Tschirnhaus transformation
and if $\cha k=3$, then one first makes a suitable change $t\mapsto t+r$
that makes the coefficient in front the factor $t$ zero, followed
by a transformation $t\to t^{-1}$.
We also have $a\neq0$ if $\cha k=3$ because the cubic is assumed
to be separable.
We let $\Delta=-4a-27b^2$ be the discriminant and 
recall that $L/K$ is a $\ZZ/3$-extension if and only if $\Delta\in(k^*)^2$,
which is also true if $\cha k=3$.

Recall from Propositions \ref{prop:BD2-1d-irreps} 
that we have an isomorphism
$$
\left(\Norm_{\SL_{2,k}}(\BD_2)/\BD_2\right)(\overline{k})\,\cong\, \Sym_3,
$$
where the unique index $2$ subgroup of the left hand side is generated by
$$
\sigma \,:=\,\frac{1}{\sqrt{2}}\begin{bmatrix}
    \zeta_8 & \zeta_8\\
    -\zeta_8^{-1} & \zeta_8^{-1}
\end{bmatrix}\in \BT^*(\overline{k})\setminus\BD_2(\overline{k})
$$
for some primitive $8$.th root of unity $\zeta_8\in\overline{k}$.
This is also true if $\cha k=3$, where this simplifies as we may assume that
$\sqrt{2}=\zeta_8^2\in\FF_9$, and we note that $\sigma$ is not diagonalisable, but
unipotent of order 3.

We also have an order $2$ element given by
$$
\tau \,:=\,\diag(\zeta_8,\zeta_8^{-1})\,\in\,
\bmu_8(\overline{k})\setminus\bmu_4(\overline{k}).
$$
\begin{itemize}
\item  If $L/k$ is a $\ZZ/3$-extension, then  
$\Gal(L/k)\to(\Norm_{\SL_2}(\BD_2)/\BD_2)(\overline{k})$ 
identifies $\Gal(L/k)$ with the subgroup generated by $\sigma$. 
Hence, letting the roots of $t^3+at+b$ be $\theta_0,\theta_1,\theta_2$, 
we may assume
$$
\sigma(\theta_i) \,=\,\theta_{i+1},
$$
where the indices are taken mod $3$. 
\item
If $L/k$ is an $\Sym_3$-extension, then 
$\Gal(L/k)\to(\Norm_{\SL_2}(\BD_2)/\BD_2)(\overline{k})$ is an isomorphism. 
In this case, $\tau$ fixes one of the roots, which we may take to be $\theta_0$. 
Hence, we may additionally assume 
$$
\tau(\theta_0)=\theta_0, \quad \tau(\theta_1)=\theta_2,\quad \tau(\theta_2)=\theta_1.
$$
\end{itemize}

Next, we have
$$
L[x,y]^{\BD_2} \,=\, L[u,v,w],
$$
where 
\begin{equation}\label{really-BD2-invariants}
u:=x^4+y^4,\quad v:=(xy)^2,\quad w:=xy(x^4-y^4).
\end{equation}
Then one computes from the explicit matrices above that
$$
\sigma(w)\,=\,w,\quad\tau(w)\,=\,-w
$$
hence, the vector space $Lw$ is the sign representation of $\Sym_3$. 
On the other hand, we compute
$$
\tau(u) =-u,\quad\tau(v)=v,\quad\sigma(u)=-\frac{1}{2}u+3v,
\quad\sigma(v)=-\frac{1}{4}u-\frac{1}{2}v.
$$
and letting
\begin{equation}\label{eqn:fiS3twist}
f_0:=-4v,\quad f_1:=u+2v,\quad f_2:=-u+2v,    
\end{equation}
we see that $\tau$ and $\sigma$ act as follows
$$
\begin{array}{cclllc}
\tau &:& f_0\mapsto f_0, & f_1\mapsto f_2, & f_2\mapsto f_1\\
\sigma &:& f_0\mapsto f_1, & f_1\mapsto f_2, & f_2\mapsto f_0&.
\end{array}
$$
Thus, for both the $f_i$ and the $\theta_i$, we see $\sigma$ acts on the 
subscripts as the $3$-cycle $(012)$, 
and when $\Gal(L/k)\cong\Sym_3$, we see $\tau$ acts as 
the transposition $(12)$.

Recall that
$$
k\quad\subseteq\quad
L^\sigma \,=\,L^{\ZZ/3} \,=\,k(\sqrt{\Delta})
$$
coincides with $k$ or it is is the unique quadratic 
extension of $k$ in $L$.

We now compute
\begin{eqnarray*}
  L[x,y]^{\BD_2\rtimes\Gal(L/k)} &=& L[u,v,w]^{\Gal(L/k)} \\
  &=&
\left\{
\begin{array}{ll}
L[f_i;w]^{\langle\sigma\rangle}& \mbox{ if }\Gal(L/k)=\ZZ/3\\
(L[f_i;w]^{\langle\sigma\rangle})^{\ZZ/2}& \mbox{ if } \Gal(L/k)=\Sym_3.
\end{array}
\right.
\end{eqnarray*}

\subsubsection{$\cha k\neq3$}
Here, we first consider the case where $\zeta_3\in L$, that is, 
$L$ contains a $3$rd root of unity.
Then, $k(\zeta_3)\subseteq L$. 
Since there is at most one quadratic extension of $k$ in $L$, 
we see either $\zeta_3\in k(\sqrt{\Delta})\setminus k$ or $\zeta_3\in k$ 
and in either case $\sigma(\zeta_3)=\zeta_3$. Consider the Lagrange resolvents
$$
g^\pm \,:=\,\sum_i\zeta_3^{\mp i} f_i,\quad 
\eta^\pm:=\sum_i\zeta_3^{\mp i} \theta_i.
$$
Observe
$$
v=-\frac{1}{12}(g^++g^-),\quad u=\frac{1}{\sqrt{-12}}(g^--g^+),
$$
and $\sqrt{-12}=2(\zeta_3-\zeta_3^{-1})\in L$, 
so equation \eqref{Dn-sing} becomes
\begin{equation}\label{D4-twist0}
108w^2 \,=\, (g^+)^3+(g^-)^3.    
\end{equation}

Let us first make some basic observations about $\eta^\pm$. 
One checks
\begin{equation}\label{etapmequs1}
\eta^+\eta^-=-3a,\quad (\eta^+)^3+(\eta^-)^3=-27b,    
\end{equation}
which implies that $(\eta^\pm)^3$ are the roots of 
$$
t^2 + 27bt -27a^3.
$$
Additionally, we see
$$
\left((\eta^+)^3-(\eta^-)^3\right)^2
\,=\,\left((\eta^+)^3+(\eta^-)^3\right)^2-4(\eta^+\eta^-)^3
\,=\,-27\Delta,
$$
hence
\begin{equation}\label{etapmequs2}
(\eta^+)^3-(\eta^-)^3 \,=\,\pm3\sqrt{-3\Delta}.
\end{equation}
Furthermore, \eqref{etapmequs1} tells us at most one of the $\eta^\pm$ vanishes, 
and this occurs if and only if $a=0$: 
indeed, if both were to vanish then $a=b=0$ and our polynomial $t^3+at+b$ 
would not be irreducible.

Note
$$
\sigma(g^\pm)\,=\,\zeta_3^\pm g^\pm,\quad \sigma(\eta^\pm)\,=\,\zeta_3^\pm \eta^\pm.
$$
Hence, if $\epsilon\in\{\pm\}$ is such that $\eta^\epsilon\neq0$, we see
\begin{equation}\label{eqn:gepsinvars}
L[x,y]^{\BD_2\rtimes\Gal(L/k)} \,=\,
k(\Delta)[\eta^\epsilon g^{-\epsilon},(\eta^\epsilon)^{-1}g^\epsilon,w]^{\Gal(k(\Delta)/k)}.   
\end{equation}
If both $\eta^\pm$ are non-zero, then we also have  
\begin{equation}\label{eqn:gepsinvars2}
L[x,y]^{\BD_2\rtimes\Gal(L/k)} \,=\,k(\Delta)[\eta^\mp g^\pm,w]^{\Gal(k(\Delta)/k)}
\end{equation}
Note in the above computations, we have used $\eta^{-\epsilon}=a/\eta^\epsilon$, 
so, for example, the invariant $\eta^{-\epsilon} g^{\epsilon}$ 
is not needed as it equals $a(\eta^\epsilon)^{-1}g^\epsilon$.

Consider the subcase where $\zeta_3\in L$ and $-3\Delta\in k^{\times 2}$. 
If $L/k$ is an $\Sym_3$-extension, then this is equivalent to 
$\zeta_3\in k(\sqrt{\Delta})\setminus k$ 
and if $L/k$ is an $\ZZ/3$-extension, then this is equivalent to $\zeta_3\in k$. 
In either case, $\eta^\pm$ and $g^\pm$ are fixed by $\Gal(k(\Delta)/k)$ 
since in the latter case the Galois group is trivial, 
and in the former case the Galois group is generated by $\tau$ which acts as $\tau(\zeta_3)=\zeta_3^{-1}$. 
Thus, letting $\epsilon\in\{\pm\}$ such that $\eta^\epsilon\neq0$, we have
$$
L[x,y]^{\BD_2\rtimes\Gal(L/k)} \,=\,k[\eta^\epsilon g^{-\epsilon},(\eta^\epsilon)^{-1}g^\epsilon,\sqrt{\Delta}w].
$$
Letting
$$
A :=\eta^\epsilon g^{-\epsilon},\quad B:=(\eta^\epsilon)^{-1}g^\epsilon,\quad C:=\sqrt{\Delta}w
$$
equation \eqref{D4-twist0} becomes
\begin{equation}\label{D4-twist1a}
\Delta^{-1}108C^2 \,=\,(\eta^\epsilon)^{-3}A^3+(\eta^\epsilon)^{3}B^3.    
\end{equation}
By rescaling $C$ and using that $108/\Delta\in k^{\times 2}$ (note that
we assumed $-3\Delta\in k^{\times2}$), we arrive at an equation
\begin{equation}
\label{eq:G2case1}
C^2 \,=\, (\eta^\epsilon)^3A^3+(\eta^\epsilon)^{-3} B^3.
\end{equation}
Note that this case includes the case where $\zeta_3\in k$ and
$L/k$ is a cubic extension. 
By Kummer theory, we may assume that $L/k$ is generated by a cubic 
of the form $t^3+b=0$ and thus, $L=k(\beta)$ for some choice of 
$\beta:=\sqrt[3]{b}$.
Instead of using the Lagrange resolvents $\eta^{\pm}$, 
we see immediately 
that $\beta^\pm g^\pm$ and $w$ are invariant leading finally
to an equation of the form
\begin{equation}
C^2 \,=\, \beta^3A^3+\beta^{-3} B^3 \,=\, bA^3+b^{-1}B^3.
\end{equation}

Next, consider the subcase where $\zeta_3\in L$ and $-3\Delta\notin k^{\times2}$. 
This is equivalent to $\zeta_3\in k$ and $L/k$ an $\Sym_3$-extension. 
Then $\tau(\eta^\pm)=\eta^\mp$ and $\tau(g^\pm)=g^\mp$. 
Note in particular that $\eta^\pm\in L^*$ since, by the Galois action, 
one of $\eta^\pm$ vanishes if and only if they both vanish, 
and we have already established that at most one of $\eta^\pm$ can vanish. 
Thus, by \eqref{eqn:gepsinvars2},
$$
L[x,y]^{\BD_2\rtimes\Gal(L/k)} \,=\,k(\Delta)[\eta^\mp g^\pm,w]^{\Gal(k(\Delta)/k)}=k[A,B,C]
$$
where
\begin{equation}\label{case1bG2twist}
A:=\eta^-g^++\eta^+g^-,\quad B:=\sqrt{-3\Delta}(\eta^-g^+-\eta^+g^-),\quad C:=\sqrt{\Delta}w;    
\end{equation}
note the $\sqrt{-3}$ factor in $B$ is unnecessary but turns out to be convenient. 
Then equation \eqref{D4-twist0} becomes 
\begin{align*}
8(\eta^+\eta^-)^3\sqrt{\Delta}108C^2 &= (\eta^+)^3(2\sqrt{\Delta}\eta^-g^+)^3+(\eta^-)^3(2\sqrt{\Delta}\eta^+g^-)^3\\
&=(\eta^+)^3(\sqrt{\Delta}A+\sqrt{-3}^{-1}B)^3+(\eta^-)^3(\sqrt{\Delta}A-\sqrt{-3}^{-1}B)^3\\
&=\sqrt{-3}^{-1}B((\eta^+)^3-(\eta^-)^3)(3\Delta A^2-\frac{B^2}{3})\\
&\qquad 
+\sqrt{\Delta}A((\eta^+)^3+(\eta^-)^3)(\Delta A^2-B^2).
\end{align*}
Using equations \eqref{etapmequs1} and \eqref{etapmequs2} we obtain
\begin{equation}\label{D4-twist1b-long}
-2^5 3^6 a^3 C^2 \,=\, -3bA(9\Delta A^2-9B^2)\pm B(9\Delta A^2-B^2)
\end{equation}
Multiplying by $-1$, replacing $A$ by $A/3$, replacing $B$ by $\pm B$, 
and replacing $C$ by $C/(2^4 3^6 a^2)$
we obtain:
\begin{equation}\label{D4-twist1b}
2 a C^2 \,=\, bA(\Delta A^2-9B^2) + B(\Delta A^2-B^2)
\end{equation}
Note that over $k(\sqrt{-3\Delta})$, we can substitute $A+B$ for $A$
and $\sqrt{-3\Delta}(A-B)$ for $B$ into \eqref{D4-twist1b-long} and get
$$
\frac{2^3 3^6 a^3}{\Delta}C^2 \,=\, B^3(\sqrt{-3\Delta}+9b)-A^3(\sqrt{-3\Delta}-9b)
$$
and then using the quadratic equation that $(\eta^{\pm})^3$ satsify, we see 
$$
\pm\frac{2}{3}(\eta^\pm)^3 \,=\,\sqrt{-3\Delta}\mp 9b,
$$
so this yields
$$
\frac{2^23^6a^3}{\Delta}C^2=(\eta^+)^3B^3+(\eta^-)^3A^3
$$
and after rescaling all $A,B,C$ by $2^23^6a^3/\Delta$, we recover 
and we recover \eqref{eq:G2case1}.
Running these substitutions and rescaling backwards, 
this shows that \eqref{eq:G2case1} can be transformed into
\eqref{D4-twist1b} if $\sqrt{-3\Delta}\in k$.

Finally, we consider the case where $\zeta_3\notin L$. 
Letting $G$ be our twisted form of $\BD_2$, since $G$ is linearly reductive, we know
$$
k(\zeta_3)[x,y]^{G\times_k k(\zeta_3)} \,=\,k[x,y]^G\otimes_k k(\zeta_3).
$$
In particular,
$$
k[x,y]^G=(k(\zeta_3)[x,y]^{G\times_k k(\zeta_3)})^{\Gal(k(\zeta_3)/k)}.
$$
Now, $L(\zeta_3)$ is the compositum of $L$ and $k(\zeta_3)$, 
so $\Gal(L(\zeta_3)/k(\zeta_3))\cong\Gal(L/k)$ and 
$\Gal(L(\zeta_3)/L)\cong\Gal(k(\zeta_3)/k)$. 
Thus, if $L/k$ is an $\Sym_3$-extension, then after base changing to 
$k(\zeta_3)$, we find ourselves in the second subcase above. 
So, we see
\begin{equation}\label{eqn:G2case2}
k[x,y]^G \,=\, \left(k(\zeta_3)[A,B,C]\right)^{\Gal(k(\zeta_3)/k)},    
\end{equation}
where $A,B,C$ are as in equation \eqref{case1bG2twist}. 
Since $\sqrt{\Delta}\in L$, it is fixed by the Galois action. 
Then we see $A,B,C$ are fixed by the Galois action, so we again obtain equation 
\eqref{D4-twist1b}. 
On the other hand, if $L/k$ is a $\ZZ/3$-extension, then $\eta^\pm\neq0$ 
over $L(\zeta_3)$ since otherwise \eqref{etapmequs1} shows $a=0$, 
which is a non-Galois cubic as $\omega\notin k$ so 
$\Delta=-27b^2\notin (k^*)^2$. 
Then, after base changing to $k(\zeta_3)$ and using equation \eqref{eqn:gepsinvars2}, 
we find
$$
k[x,y]^G \,=\, \left(k(\omega)[\eta^\pm g^\mp,w]\right)^{\Gal(k(\zeta_3)/k)},    
$$
and the Galois action swaps $\eta^\pm g^\mp$. 
Thus, the invariants are again $A,B,C$ as in equation \eqref{case1bG2twist},
which satisfy \eqref{D4-twist1b}. 
Summing up, we see that we can always bring the $G_2$-singularities
that arise as $\Aff^2/G$ and in characteristic $\neq2,3$
into the form \eqref{D4-twist1b}.
In some special cases, for example, 
$\zeta_3\in k$ and $\Delta\in k^{\times 2}$, we can find simpler forms
such as \eqref{eq:G2case1}.

\subsubsection{$\cha k=3$}
In this case, we have $\Delta=-a\neq0$ for otherwise, the extension $L/k$ 
becomes trivial or inseparable.
Moreover, the $\Sym_3$-action on $L[u,v,w]$ simplifies to
$$
\begin{array}{cclll}
\tau &:& u\mapsto -u, & v\mapsto v, & w\mapsto -w\\
\sigma &:& u\mapsto u, & v\mapsto v-u, & w\mapsto w
\end{array}
$$

First, assume that $L/k$ is a $\ZZ/3$-extension.
This is equivalent to $-a=\Delta$ being a square in $k^{\times}$.
Setting $t'=t\sqrt{-a}$, we obtain that $L/k$ is
generated by a root $\beta\in L$ of the
Artin--Schreier equation 
$$
  \wp(t') \,:=\, t'^3-t'=\frac{b}{a\sqrt{-a}}
$$
and the other two roots are $\beta\pm1$ with $\sigma(\beta)=\beta+1$. 
We find the invariants $u,w, v+\beta u$ for the diagonal
$\sigma$-action on $L[u,v,w]$ and they obey an
equation
$$
(v+\beta u)^3\,=\,-w^2+u^2v+\beta^3u^3\,=\,
-w^2+u^2(v+\beta u)+(\beta^3-\beta)u^3.
$$
Thus, setting $U=u$, $V=v+\beta u$, and $W=w$,
we find
$$ 
 -W^2+ U^2V + \frac{b}{a\sqrt{-a}} U^3\,=\, V^3.
$$
Replacing $U$ by $-U$ and $V$ by $-V$, we find
\begin{equation}
\label{eq:G2char3Z3}
 W^2+ U^2V + \frac{b}{a\sqrt{-a}} U^3\,=\, V^3.
\end{equation}

Second, assume that $L/k$ is an $\Sym_3$-extension,
that is, $-a\not\in k^{\times2}$.
Then we have $k\subseteq k(\sqrt{-a})\subseteq L$.
Over $k(\sqrt{-a})$, the singularity becomes isomorphic to
\eqref{eq:G2char3Z3}.
We have $\tau(\sqrt{-a})=-\sqrt{-a}$, which implies
that $\wp(\beta)=-\wp(\beta)$, which implies that 
$\tau(\beta)=-\beta+c$ for some $c\in\FF_3$.
Replacing $\beta$ by $\beta+c$, we may assume
$\tau(\beta)=-\beta$ and let $V=v+\beta u$ with respect
to this $\beta$.
Then, $\tau$ acts as
$U\mapsto -U$, $V\mapsto V$, and $W\mapsto -W$
and we have the invariants
$$
U':=\sqrt{-a}U,\quad  V,\quad W':=\sqrt{-a}W
$$
that satisfy the equation
$$
W'^2+U'^2V-\frac{b}{a^2}U'^3=-aV^3.
$$
Rescaling $W\mapsto W/a$, $V'\mapsto V'/(-a)$, this becomes
\begin{equation}
\label{eq:G2char3S3}
W'^2  - aU'^2V-bU'^3=V'^3.
\end{equation}
Note that if $-a$ is a square in $k^\times$, then this simplifies
to \eqref{eq:G2char3Z3} by rescaling $U''=\sqrt{-a}U'$.
In the degenerate case $b=0$ (the cubic $t^3+at+b$ is not irreducible), 
this becomes a quadratic twist with respect to $k(\sqrt{-a})/k$,
which we deal with in the next section.

\subsection{Types $C_3$ and $D_4$}\label{subsec:D4quadtwist}
Finally, we compute the quadratic twists of the $D_4$-singularity. 
Let $k$ be a field with $\cha k\neq2$ , let
$L=k(\sqrt{d})$ be a quadratic extension of $k$,
and consider an injection
$$
\Gal(L/k) \,\to\, (\Norm_{\SL_{2,k}}(\BD_2)/\BD_2)(\overline{k})\,\cong\,\Sym_3.
$$
Then
$$
L[x,y]^{\BD_2\rtimes\Gal(L/k)} \,=\,L[u,v,w]^{\Gal(L/k)}\,=\,L[f_i;w]^{\Gal(L/k)}
$$
where $u,v,w$ and the $f_i$ are as in equations \eqref{really-BD2-invariants} 
and \eqref{eqn:fiS3twist}. 
By the description of $(\Norm_{\SL_2}(\BD_2)/\BD_2)(\overline{k})$ given in 
\S\ref{subsec:S3twist}, the generator $\tau'$ of $\Gal(L/k)$ 
acts as an $\Sym_3$-conjugate of $\tau$, that is, 
there exists $\epsilon\in\Sym_3$ such that
$$
\tau'(f_{\epsilon(0)})=f_{\epsilon(0)},\quad \tau'(f_{\epsilon(1)})=f_{\epsilon(2)}.
$$
Thus,
$$
L[x,y]^{\BD_2\rtimes\Gal(L/k)} \,=\,k[A,B,C],
$$
where
$$
A:=f_{\epsilon(0)},\quad B:=\sqrt{d}\left(f_{\epsilon(1)}-f_{\epsilon(2)}\right),
\quad C:=2\sqrt{d}w.
$$
Note we have used $\sum_i f_i=0$ to conclude $L[A,B]=L[f_i]$. 
Then, \eqref{Dn-sing} becomes
$$
C^2 \,=\, d\prod_i f_i.
$$
Since $A=f_{\epsilon(0)}=-(f_{\epsilon(1)}+f_{\epsilon(2)})$, we see
$$
4f_{\epsilon(1)}f_{\epsilon(2)} \,=\,A^2-d^{-1}B^2.
$$
Thus,
$$
4C^2 \,=\, A(dA^2-B^2).
$$
Replacing $C$ with $2C$ and $A$ with $-A$, we see the above equation 
is the same as the $C_3$-singularity equation in 
Theorem \ref{tm:twisted-invariants} 
after replacing $X$ and $Y$ with $2X$ and $2Y$.
Note however, that 
the computations in \S\ref{subsec:CDtwist} correspond
to the computations in this section in the special case, where 
the injection 
$\Gal(L/k)\to\Sym_3$ has image equal to the subgroup
generated by $\tau$, which covers only one of the three
subgroups of order 2 inside $\Sym_3$.

\subsection{Type $E_7$}
Let $k$ be a field with $\cha k\neq2,3$.
We have $\BT^*$ is a normal subgroup scheme of $\BO$ of index 2
and thus
$$
k[x,y]^\BO\,=\,(k[x,y]^{\BT^*})^{\ZZ/2}\,\cong\,
(k[C,D,E]/(D^2+108\,C^4-E^3))^{\ZZ/2},
$$
where $C=xy(x^4-y^4)$, 
$D=x^{12}+y^{12}-33x^4y^4(x^4+y^4)$,
and $E=x^8+14x^4y^4+y^8$, and where
the $\ZZ/2$-action on them is given by 
$$
  C\mapsto -C, \quad D\mapsto -D,\quad E\mapsto E,
$$
see \S\ref{subsec:S3twist}.
Setting $R:=C^2$, $S:=D^2$, $T:=CD$,
we see that $R,S,T,E$ are invariants.
We have $RS=T^2$ and since 
$S=D^2=E^3-108C^4=E^3-108R^2$, 
we can eliminate $S$ from the invariants.
We find
$$
k[x,y]^\BO\,=\,k[R,T,E]/(T^2+108\,R^3-RE^3),
$$
%where $R=C^2=x^2y^2(x^8-2x^4y^4+y^8)$,
%$T=CD=xy(x^{16}-34x^{12}y^4+34x^4y^{12}-xy^{16})$,
%and $E=x^8+14x^4y^4+y^8$.
which is isomorphic to 
$$
k[R',S',T']/(T'^2+4R'^3-R'E'^3)
$$
after suitable rescaling.
This is a rational double point of type $E_7$.

By Theorem \ref{thm:NGmodGAutMcKay-more-general} and 
Corollary \ref{cor:twisted-invariants-NGmodG}, there are no
non-trivial twisted forms of the RDP of type $E_7$
over perfect fields.

\subsection{Type $E_8$}
Let $k$ be a field with $\cha k\neq2,3,5$ and
let $\zeta_5\in\overline{k}$ be a primitive $5$.th root of
unity.
Consider the action of $r,t$ on $k(\zeta_5)[x,y]$
as described in Proposition \ref{prop:BI-descends}.
Then,
$$
f\,:=\, xy(x^{10}-11xy-y^{10})
$$
is an invariant, see \cite[Kapitel II.\S13]{Klein}.
(Note that our action of $\BI$ is slightly different
from Klein's, which leads to different signs here and there.)
Following loc.cit., we consider the Hessian of $f$
\begin{eqnarray*}
H &:=& \frac{1}{121}\det
\begin{bmatrix}
    \frac{\partial^2}{\partial x^2}f & \frac{\partial^2}{\partial x\partial y}f \\
      \frac{\partial^2}{\partial x\partial y}f & \frac{\partial^2}{\partial y^2}f   
\end{bmatrix} \\
&=&
-(x^{20}+y^{20})-228(x^{15}y-x^5y^{15})-494x^{10}y^{10}
\end{eqnarray*}
(if $\cha k=11$, one ignores the determinant 
and takes the second line as definition of $H$) and check that it is an
invariant, as well as
\begin{eqnarray*}
  T &:=& \frac{1}{20}\det
\begin{bmatrix}
    \frac{\partial}{\partial x}f & \frac{\partial}{\partial y}f \\
      \frac{\partial}{\partial x}H & \frac{\partial}{\partial y}H   
\end{bmatrix}  \\
&=& (x^{30}+y^{30})-522(x^{25}y^5-x^5y^{25})-10005(x^{20}y^{10}+x^{10}y^{20}),
\end{eqnarray*}
which is also an invariant.
These satisfy the equation
$$
T^2+H^3+1728 f^5 \,=\,0
$$
and after rescaling $f,T,H$, these satisfy the 
equation $T'^2+H'^3+f'^5=0$.
We also note that the invariants their algebraic dependence
descend from $k(\zeta_5)$ to $k$.

By Theorem \ref{thm:NGmodGAutMcKay-more-general} and 
Corollary \ref{cor:twisted-invariants-NGmodG}, there are no
non-trivial twisted forms of the RDP of type $E_8$
over perfect fields.

\end{document}